\documentclass[12pt,reqno]{amsart}
\usepackage{amsfonts,amsmath,amssymb,graphicx,array,float}
\usepackage{verbatim}

\usepackage{tikz}
\usetikzlibrary{arrows}
\usetikzlibrary{fadings}
\usetikzlibrary{decorations.shapes}

\tikzset{decorate sep/.style 2 args=
{decorate,decoration={shape backgrounds,shape=circle,shape size=#1,shape sep=#2}}}

\usepackage{hyperref}

\theoremstyle{plain} \newtheorem{Thm}{Theorem}[section]
\theoremstyle{plain} \newtheorem{Cor}[Thm]{Corollary}
\theoremstyle{plain} \newtheorem{Prop}[Thm]{Proposition}
\theoremstyle{plain} \newtheorem{Lemma}[Thm]{Lemma}
\theoremstyle{definition} 
\theoremstyle{definition} \newtheorem{Rem}[Thm]{Remark}
\theoremstyle{definition} 
\theoremstyle{definition} \newtheorem{Ex}[Thm]{Example}

\newcommand{\thmlist}{
\renewcommand{\theenumi}{\alph{enumi}}
\renewcommand{\labelenumi}{(\theenumi)}}
\renewcommand{\Re}{\mathop{\rm{Re}}}
\renewcommand{\Im}{\mathop{\rm{Im}}}

\newcommand{\id}{\mathop{\rm{id}}}

\newcommand{\ad}{\mathop{\rm{ad}}}

\newcommand{\Ker}{\mathop{\rm{ker}}}

\newcommand{\End}{\mathop{\rm{End}}}

\newcommand{\inner}[2]{\langle#1,#2\rangle}

\newcommand{\C}{\ensuremath{\mathbb C}}

\newcommand{\D}{\ensuremath{\mathbb D}}

\newcommand{\R}{\ensuremath{\mathbb R}}
\newcommand{\Z}{\ensuremath{\mathbb Z}}
\newcommand{\N}{\ensuremath{\mathbb N}}

\renewcommand{\l}{\lambda}
\renewcommand{\a}{\alpha}
\renewcommand{\b}{\beta}

\newcommand{\fa}{\mathfrak{a}}

\newcommand{\fm}{\mathfrak{m}}

\newcommand{\la}{\l_\a}

\newcommand\SO{\mathop{\rm{SO}}}

\newcommand\SU{\mathop{\rm{SU}}}
\newcommand\U{\mathop{\rm{U}}}
\newcommand\Sg{\mathop{\rm{S}}}
\newcommand\Sp{\mathop{\rm{Sp}}}
\newcommand\Spin{\mathop{\rm{Spin}}}

\newcommand{\frakacs}{\mathfrak a_{\C}^*}

\newcommand{\polya}{{\rm S}(\mathfrak{a}_\C)}

\newcommand{\rml}{{\rm l}}
\newcommand{\rmm}{{\rm m}}
\newcommand{\rms}{{\rm s}}

\newcommand{\tauwell}{(\ell,\well)}

\newcommand{\wt}{\widetilde}
\newcommand{\well}{\widetilde{\ell}}
\renewcommand{\phi}{\varphi}

\begin{document}

\makeatletter
\title[Hypergeometric functions of type $BC$]{Hypergeometric functions of type $BC$\\ and standard multiplicities}

\author{E. K. Narayanan}
\address{Department of Mathematics, Indian Institute of Science, Bangalore-12, India.}
\email{{naru@iisc.ac.in}}

\author{A. Pasquale}
\address{Institut Elie Cartan de Lorraine (UMR CNRS 7502),
Universit\'{e} de Lorraine, F-57045 Metz, France.}
\email{angela.pasquale@univ-lorraine.fr}

\thanks{The authors thank both the referees for a careful reading of the manuscript and pointing out some mistakes. The authors gratefully acknowledge financial support by the
Indo-French Centre for the Promotion of Advanced Research (IFCPAR/CEFIPRA),
Project No. 5001--1: ``Hypergeometric functions: harmonic analysis and representation theory''. The first named author thanks SERB, India for the financial support through MATRICS grant MTR/2018/00051.}

\date{}

\begin{abstract}
We study the Heckman-Opdam hypergeometric functions associated to a root system of type $BC$ and a multiplicity function which is allowed to assume some non-positive values (a standard multiplicity function). For such functions, we obtain positivity properties and sharp estimates  which imply a characterization of the bounded hypergeometric functions. As an application, our results extend known properties of Harish-Chandra's spherical functions on Riemannian symmetric spaces of the non-compact type $G/K$ to spherical functions over homogeneous vector bundles on $G/K$ which are associated to certain small $K$-types.
\end{abstract}

\maketitle

\tableofcontents
\section*{Introduction}

Let $G$ be a connected non-compact real semisimple Lie group with finite center, $K$ a maximal compact subgroup of $G$, and $X=G/K$ the corresponding Riemannian symmetric space of the non-compact type. Harish-Chandra's theory of spherical functions on $X$ has been extended into two different directions. 

The first direction is Heckman-Opdam's theory of hypergeometric functions on root systems. It originated from the fact that, by restriction to a maximally flat subspace of $X$, 
the $K$-invariant functions on $X$ become Weyl-group-invariant functions on the corresponding Cartan subspace $\mathfrak{a}$.
The analogue of Harish-Chandra's spherical functions are the hypergeometric functions associated with root systems, also known as Heckman-Opdam's hypergeometric functions. Heckman and Opdam introduced them in the late 1980s and, in the middle 1990s, Cherednik contributed in simplifying their original definition; see \cite{HS94,OpdamActa}, as well as  \cite{Anker} for a recent overview of the theory. In the context of special functions associated with root systems, the symmetric space $X$ is replaced by a triple $(\mathfrak{a},\Sigma,m)$, where $\mathfrak{a}$ is a Euclidean real vector space (playing the role of the Cartan subspace of $X$), $\Sigma$ is a 
root system in the real dual space $\mathfrak{a}^*$ of $\mathfrak{a}$, and $m:\Sigma\to \C$ is a multiplicity function, i.e. a function on $\Sigma$ which is invariant with respect to the Weyl group $W$ of $\Sigma$. If $\mathfrak{a}$ is a Cartan subspace of the symmetric space $X$, 
$\Sigma$ is the corresponding system of (restricted) roots and the multiplicity function $m$ is given by the dimensions of the root spaces, then Heckman-Opdam's hypergeometric functions associated with $(\mathfrak{a},\Sigma,m)$ are precisely the restrictions to $\mathfrak{a}$ of Harish-Chandra's spherical functions. Here and in the following, we adopt the convention of identifying the Cartan subspace $\mathfrak{a}$ with its diffeomorphic image $A=\exp(\mathfrak{a})$ inside $G$.
 
Notice that, even if the multiplicity function in a triple $(\mathfrak{a},\Sigma,m)$ might be generally complex-valued, most of the properties of Heckman-Opdam's hypergeometric functions and also the harmonic analysis associated with them, are known assuming that the multiplicity functions have values in $[0,+\infty)$.
This is of course the most natural setting to include Harish-Chandra's spherical harmonic analysis on Riemannian symmetric spaces.

The second direction extending Harish-Chandra's theory of spherical functions concerns the so-called $\tau$-spherical functions, where $\tau$ is a unitary irreducible representation of $K$. These functions already appeared in the work of Godement \cite{Go52} and Harish-Chandra \cite{HC51a, HC51b}, and they agree with 
Harish-Chandra's spherical functions on $X$ when $\tau$ is the trivial representation. 
They have been studied either in the context of the representation theory of $G$ or in relation to  the harmonic analysis on homogeneous vector bundles over $X$. Among the references related to the present paper, we mention \cite{HC60, Tak74, ShimenoPlancherel, HS94, O94, Ca97, Ped98, DP, OS17, RS18, HOS}. 
Several new features appear when $\tau$ is non-trivial. For instance, the algebra $\mathbb{D}(G/K;\tau)$ of invariant differential operators acting on the sections of the homogeneous vector bundle, might not be commutative. 
Let $V_\tau$ be the space of $\tau$ and let $L^1(G//K;\tau)$ denote the convolution algebra of $\End(V_\tau)$-valued functions $f$ on $G$ satisfying $f(k_1gk_2)=\tau(k_2^{-1})f(g)\tau(k_2^{-1})$ for all $g\in G$ and $k_1, k_2\in K$. It is a classical result that 
$\mathbb{D}(G/K;\tau)$ is commutative if and only if $L^1(G//K;\tau)$ is commutative. In this case, $(G,K,\tau)$ is said to be a Gelfand triple. A convenient criterion to check whether
$(G,K,\tau)$ is a Gelfand triple is due to Deitmar \cite{D90}. Namely,  $\mathbb{D}(G/K;\tau)$ is commutative if and only if the restriction of $\tau$ to $M$, the centralizer in $K$ of the Cartan subspace $\mathfrak{a}$, is multiplicity free.
This happens for instance when $\tau|_M$ is irreducible, i.e. $\tau$ is a small $K$-type. For example, a $1$-dimensional representation $\tau$ is necessarily a small $K$-type. 
For Gelfand triples, the theory of $\tau$-spherical functions can be set up exactly as in the case of $\tau$ trivial. For instance, the $\tau$-spherical functions can be equivalently defined either as joint eigenfunctions of $\mathbb{D}(G/K;\tau)$ or as characters of $L^1(G//K;\tau)$. Nevertheless, they are much more difficult to handle than in the trivial case, and many of their properties are still not known. 

The present paper is situated at a crossing of the two directions of extensions of Harish-Chandra's theory of spherical functions on $X$ mentioned above. It finds its motivations in Shimeno's paper \cite{ShimenoPlancherel}, in Heckman's chapter \cite[Chapter 5]{HS94} and, more generally, in the recent paper \cite{OS17} by Oda and Shimeno
on $\tau$-spherical functions corresponding to small $K$-types. When $\tau$ is a small $K$-type, a $\tau$-spherical function is uniquely determined by its restriction to a Cartan subspace $\mathfrak{a}$ of $G$. By Schur's lemma, this restriction is of the form $\varphi \cdot \id$ where 
$\id$ is the identity on $V_\tau$ and $\varphi$ is function on $\mathfrak{a}$ which is Weyl-group 
invariant. 
The main theorem of \cite{OS17} proves that, up to multiplication by an explicit non-vanishing smooth $\cosh$-like factor, the function $\varphi$ is a Heckman-Opdam's hypergeometric function. It  
corresponds to a triple $(\mathfrak{a},\Sigma(\tau),m(\tau))$ in which $\Sigma(\tau)$ 
 is possibly not the root system associated with the symmetric space $X$ and 
$m(\tau)$ need not be positive.  
This motivates a systematic study of Heckman-Opdam's hypergeometric functions corresponding to multiplicity functions which may assume negative values. 

In this paper we take up this line of investigations in the case of the so-called standard multiplicities. 
More precisely, let $\Sigma$ be a root system in $\mathfrak{a}^\ast$ of type $BC_r$, where 
$r$ is the dimension of $\mathfrak{a}$, and let $m = (m_\rms, m_\rmm, m_\rml)$ be a multiplicity function on $\Sigma$. Set
$$
\mathcal{M}_1=\{(m_\rms,m_\rmm,m_\rml)\in \mathcal{M}:  m_\rmm> 0, m_\rms> 0, m_\rms+2m_\rml> 0\}\,.
$$
In \cite[Definition 5.5.1]{HS94}, the elements of $\mathcal{M}_1$ were called 
the \textit{standard multiplicities}.
Standard multiplicities have been introduced by Heckman and their definition is linked to the regularity of Harish-Chandra's $c$-function; see \cite[Lemma 5.5.2]{HS94} and \eqref{eq:b-nonzero} below. 

All positive multiplicity functions are standard, but standard multiplicites also allow negative values for the long roots. Because of this, the arguments leading to the known properties of the hypergeometric functions corresponding to positive multiplicity functions do not apply to this case. By suitable modifications of their proofs, we show in this paper that 
various results, including positivity properties and estimates for hypergeometric functions, extend to standard multiplicities (see Proposition \ref{prop:positivity-estimates-M3}, Theorem \ref{thm:real} and Theorem \ref{thm:bddhyp}). 

Let $\mathfrak{a}_\C^*$ denote the complex dual vector space of $\mathfrak a$ and consider the Heckman-Opdam's (symmetric and non-symmetric) hypergeometric functions $F_\lambda(m,x)$ and $G_\lambda(m,x)$ with $\lambda \in \mathfrak{a}_\C^*$; see subsection \ref{subsection:hyp} for the definitions.
Under the condition that $m\in \mathcal{M}_1$ we prove:

\begin{enumerate}
\setlength{\itemsep}{1mm}
\item $F_\lambda$ and $G_\lambda$ are real and strictly positive on $\mathfrak{a}$ for all $\l\in \mathfrak{a}^*$,
\item $|F_\lambda|\leq F_{\Re\lambda}$ and $|G_\lambda|\leq G_{\Re\lambda}$
on $\mathfrak{a}$ for all $\l\in \mathfrak{a}^*_\C$,
\item $\max\{|F_\lambda(x)|, |G_\lambda(x)|\}\leq \sqrt{|W|} e^{\max_{w\in W}(w\l)(x)}$ for all $\lambda \in \mathfrak{a}^*_\C$ and $x\in
\mathfrak{a}$,
\item $F_{\lambda+\mu}(x)\leq F_{\mu}(x) e^{\max_{w\in W}(w\l)(x)}$
and $G_{\lambda+\mu}(x)\leq G_{\mu}(x) e^{\max_{w\in W}(w\l)(x)}$ for all
$\lambda \in \mathfrak{a}^*$, $\mu \in \overline{(\mathfrak{a}^*)^+}$ and $x\in
\mathfrak{a}$,
\item asymptotics on $\mathfrak{a}$ for $F_\lambda$ when $\lambda\in  \mathfrak{a}^*_\C$ is fixed but not necessarily regular,
\item sharp estimates on $\mathfrak{a}$  for $F_\lambda$ when $\lambda \in \mathfrak{a}^*$ is fixed but not necessarily regular.
\end{enumerate}

The estimates (but not the asymptotics) pass by continuity to the boundary of
$\mathcal{M}_1$ as well. The above properties extend to $\mathcal{M}_1$ the corresponding properties proved for nonnegative multiplicities by Opdam \cite{OpdamActa}, Ho and \'Olafsson \cite{HoOl14}, Schapira \cite{Sch08}, R\"osler, Koornwinder and Voit  \cite{RKV13}, and the authors and Pusti \cite{NPP}. 

The set of real-valued multiplicities on which both function $F_\lambda$ and $G_\lambda$ are naturally defined is
$$
\mathcal{M}_0=\{(m_\rms,m_\rmm,m_\rml)\in \mathcal{M}: m_\rmm\geq 0, m_\rms+m_\rml\geq 0\}\,.
$$
Clearly, $\mathcal{M}_1\subset \mathcal{M}_0$. In section \ref{section:basic} we introduce and discuss other subsets of $\mathcal{M}$. For instance, the estimates (3) hold in fact for a larger set of multiplicities than $\mathcal{M}_1$; see Lemma \ref{lemma:OpdamM2M3}.

Let $\rho(m)\in \mathfrak{a}^*$ be the half-sum of the positive roots in $\Sigma$, 
counted with their multiplicities; see \eqref{eq:rho}. Moreover, let $C(\rho(m))$ 
denote the convex hull of the finite set $\{w\rho(m):w\in W\}$.
Suppose first that the triple $(\mathfrak{a},\Sigma,m)$ is associated with a symmetric space $X$
and consider Harish-Chandra's parametrization of the spherical functions by the elements of
$\mathfrak{a}_\C^*$. Let $\varphi_\lambda$ denote the spherical function corresponding to 
$\l\in \mathfrak{a}_\C^*$. 
The spherical functions which are bounded can be identified with the elements of the spectrum of the (commutative) convolution algebra of $L^1$ $K$-invariant functions on $X$. They have been determined by the celebrated theorem of Helgason and Johnson \cite{HeJo}. It states that the spherical function $\varphi_\l$ on $X$ is bounded if and only 
if $\l$ belongs to the tube domain in $\mathfrak a_\C^*$ over $C(\rho)$. It is a fundamental result in the $L^1$ spherical harmonic analysis on $X$. For instance, it implies that the spherical transform of a $K$-invariant $L^1$ function on $X$ is holomorphic in the interior
of the tube domain over $C(\rho)$. In particular, the spherical transform of an $L^1$ function cannot have compact support, unlike what happens for the classical Fourier transform.

The theorem of Helgason and Johnson has been recently extended to Heckman-Opdam's hypergeometric functions in \cite{NPP} under the assumption that the multiplicity function 
$m$ is positive. On the boundary of $\mathcal{M}_1$ there are nonzero multiplicities $m$ for which $\rho(m)=0$. However, $\rho(m)\neq 0$ for $m\in \mathcal{M}_1$. It is therefore 
natural to ask if Helgason-Johnson's characterization of boundedness by means of the
tube domain over $C(\rho)$ holds for Heckman-Opdam's hypergeometric functions associated with arbitrary standard multiplicity functions $m\in \mathcal{M}_1$ and not only to the positive ones. The answer to this question is positive and is given by Theorem \ref{thm:bddhyp}.

Let $\mathcal{M}_+$ denote the set of nonnegative multiplicity functions on a fixed root system of type $BC$.
In section \ref{section:applications} we consider the 2-parameter deformations of $m=(m_\rms,m_\rmm,m_\rml)\in \mathcal{M}_+$ of the form
$$
m(\ell, \well)=(m_\rms+2\ell, m_\rmm + 2\well,m_\rml-2\ell)\,,
$$
where $\ell,\well\in \R$. The corresponding non-symmetric and symmetric $(\ell,\well)$-hypergeometric functions are respectively defined for $\lambda\in \mathfrak{a}_\C^*$ and $x\in \mathfrak{a}$ by 
\begin{eqnarray*}
G_{\ell,\well,\lambda}(m;x)&=&u(x)^{-\ell} v(x)^{-\well} G_\lambda(m(\ell,\well);x)\,, \\
F_{\ell,\well,\lambda}(m;x)&=&u(x)^{-\ell} v(x)^{-\well} F_\lambda(m(\ell,\well);x)\,, 
\end{eqnarray*}
where $u(x)$ and $v(x)$ are suitable products of hyperbolic cosine functions depending on the roots; see 
\eqref{eq:u} and \eqref{eq:v}. Given $m=(m_\rms,m_\rmm,m_\rml)\in \mathcal{M}_+$, then $m(\ell,\well)\in \mathcal{M}_1$ if and only if $-\frac{m_\rms}{2}< \ell <\frac{m_\rms}{2}+m_\rml$ and $\well> -\frac{m_\rmm}{2}$. For these values, the results proved in section \ref{section:basic} for 
$G_\lambda(m(\ell,\well))$ and $F_\lambda(m(\ell,\well))$ extend to corresponding results for the 
$(\ell,\well)$-hypergeometric functions, but some care is needed to take the factors $u^{-\ell}$ and $v^{-\well}$ into account. In particular, Theorem \ref{thm:bdd} characterizes the $F_{\ell,\well,\lambda}$'s which are bounded. This theorem as well as other properties extends to all $\ell$'s for which 
$\big| \ell-(m_\rml-1)/2\big| < (m_\rms+m_\rml+1)/2$ thanks to the symmetry relation 
$F_{\ell,\well,\lambda}=F_{-\ell+m_\rml-1,\well,\lambda}$ proved in \eqref{symmetry}.

The final section of this paper, section \ref{section:geometric}, is devoted to geometric examples based on \cite{HS94,ShimenoPlancherel} and \cite{OS17}. Namely, we prove that the $\tau$-spherical functions on $G/K$ associated with a small $K$-type
$\tau$ and for which the root system of $G/K$ is of type $BC$ can be described as symmetric $(\ell,\well)$-hypergeometric functions on the root system $\Sigma(\tau)$ of \cite{OS17} and specific choices of a multiplicity function $m\in \mathcal{M}_+$ and of the deformation parameters $(\ell,\well)$. Our general results from section \ref{section:applications} apply therefore to these cases. 

\section{Notation and preliminaries} \label{section:notation}

In this section we collect the basic notation and some preliminary results. We refer to \cite{HS94} and \cite{NPP} for a more extended exposition.
\smallskip

If $F(m)$ is a function on a space $X$ that depends on a parameter $m$, we will denote its value at $x\in X$ by $F(m;x)$ rather than $F(m)(x)$.
Given two nonnegative functions $f$ and $g$
on a same domain $D$, we write $f \asymp g$ if there exist
positive constants $C_1$ and $C_2$ so that $C_1 \leq
\frac{f(x)}{g(x)} \leq C_2$ for all $x \in D$.

\subsection{Root systems}
\label{subsection:roots}

Let $\mathfrak{a}$ be an $r$-dimensional real Euclidean vector space
with inner product $\inner{\cdot}{\cdot}$, and let $\mathfrak{a}^*$ be
the dual space of $\mathfrak{a}$. For $\l\in \mathfrak{a}^*$ let $x_\l \in
\mathfrak a$ be determined by the condition that
$\l(x)=\inner{x}{x_\l}$ for all $x \in \mathfrak a$. The assignment
$\inner{\l}{\mu}:=\inner{x_\l}{x_\mu}$ defines an inner product in
$\mathfrak{a}^*$. Let $\mathfrak{a}_\C$ and $\mathfrak{a}_\C^*$ respectively
denote the complexifications of $\mathfrak{a}$ and $\mathfrak{a}^*$. The
$\C$-bilinear extension to $\mathfrak{a}_\C$ and $\mathfrak{a}_\C^*$ of the
inner products on $\mathfrak{a}^*$ and $\mathfrak{a}$ will also be indicated
by $\inner{\cdot}{\cdot}$.
We denote by  $|\cdot|=\inner{\cdot}{\cdot}^{1/2}$ the associated norm.
We shall often employ the notation
\begin{equation}\label{eq:la}
  \la:=\frac{\inner{\l}{\a}}{\inner{\a}{\a}}
\end{equation}
for elements $\l,\a\in \fa_\C^*$ with $|\a|\neq 0$.

Let $\Sigma$ be a root system in $\mathfrak{a}^*$
with set of positive roots of the form
$\Sigma^+=\Sigma^+_\rms \sqcup \Sigma^+_\rmm \sqcup \Sigma^+_\rml$, where
\begin{equation}
\label{eq:roots}
\Sigma^+_\rms=\big\{\frac{\beta_j}{2}: 1\leq j \leq r\big\}\,,  \notag\, \quad
\Sigma^+_\rmm=\big\{\frac{\beta_j\pm \beta_i}{2}: 1\leq i < j \leq r\big\}\,,\quad
\Sigma^+_\rml=\{\beta_j: 1\leq j \leq r\} \,.
\end{equation}
The positive Weyl chamber $\mathfrak{a}^+$ consists of the elements  $x\in\mathfrak{a}$ for which
$\a(x)>0$ for all $\a \in \Sigma^+$; its closure is
$\overline{\mathfrak{a}^+}=\{x \in \mathfrak{a}: \text{$\a(x) \geq 0$ for all $\a \in \Sigma^+$}\}$.
Dually, the positive Weyl chamber $(\mathfrak{a}^*)^+$ consists of the elements  $\l\in\mathfrak{a}^*$ for which
$\inner{\l}{\a}>0$ for all $\a \in \Sigma^+$. Its closure is denoted $\overline{(\mathfrak{a}^*)^+}$.

We assume that the elements of $\Sigma^+_\rml$ form an orthogonal basis of $\fa^*$ and that they all have the same norm $p$.
We denote by $W$ the Weyl group of $\Sigma$. 
It acts on the roots by permutations and sign changes. For  $a\in \{\rms,\rmm,\rml\}$ set $\Sigma_a=\Sigma^+_a \sqcup (-\Sigma^+_a)$.  
A multiplicity function on $\Sigma$ is a $W$-invariant function on $\Sigma$. It is therefore given by a triple $m=(m_\rms,m_\rmm,m_\rml)$ of complex numbers so that $m_{a}$ is the (constant) value of $m$ on $\Sigma_{a}$ for ${a}\in \{\rms,\rmm,\rml\}$.

It will be convenient to refer to a root system $\Sigma$ as above as of type $BC_r$ even if some values of $m$ are zero. This means that root systems of type $C_r$ will be considered as being of type $BC_r$ with $m_\rms=0$ and $m_\rmm\neq 0$. Likewise, the direct products of rank-one root systems $(BC_1)^r$ and $(A_1)^r$ will be considered of type $BC_r$ with $m_\rmm=0$ and with $m_\rmm=m_\rms=0$, respectively.

The dimension $r$ of $\mathfrak{a}$ is called the \emph{(real) rank} of
the triple $(\mathfrak{a}, \Sigma, m)$. Finally, we set
\begin{equation}
\label{eq:rho}
 \rho(m)=\frac{1}{2} \sum_{\a \in \Sigma^+} m_\a \a=\frac{1}{2}  \sum_{j=1}^r \Big(\frac{m_\rms}{2}+m_\rml+(j-1)m_\rmm\Big) \beta_j \in \fa^*\,.
\end{equation}

\begin{Ex}[Geometric multiplicities]
\label{rem:hermitian}
For special values of the multiplicities $m=(m_\rms,m_\rmm,m_\rml)$, triples $(\mathfrak{a}, \Sigma, m)$ as above appear as restricted root systems of a large family of irreducible symmetric spaces $G/K$ of the non-compact type.  Among these spaces, a remarkable class consists of the noncompact Hermitian symmetric spaces; see \cite{He1} or \cite{ShimenoPlancherel}. For these spaces $m_\rml=1$.  Their root systems and multiplicity functions are listed in Table \ref{table:Hermitian}. The literature on Hermitian symmetric spaces refers to the situations where $\Sigma=C_r$ and $\Sigma=BC_r$ as to the \textit{Case I} and the \textit{Case II}, respectively.

\begin{table}[h]
\setlength{\extrarowheight}{2pt}
\begin{tabular}{|c|c|c|c|}
\hline
$G$ &$K$ & $\Sigma$ &$(m_\rms,m_\rmm,m_\rml)$ \\[.2em]
\hline\hline
$\SU(p,q)$ & $\Sg(\U(p)\times \U(q))$ &
\begin{tabular}{l}$p=q$: $C_p$ \\$p<q$: $BC_p$ \end{tabular}
&\begin{tabular}{l} $p=q$: $(0,2,1)$\\ $p<q$: $(2(q-p),2,1)$\end{tabular}  \\[.2em]
\hline
$\SO_0(p,2)$ &$\SO(p)\times\SO(2)$ & $C_2$ &$(0,p-2,1)$ \\[.2em]
\hline
$\SO^*(2n)$ & $\U(n)$ & \begin{tabular}{l} $n$ even: $C_n$\\
$n$ odd: $BC_n$ \end{tabular}  & \begin{tabular}{l} $n$ even: $(0,4,1)$\\ $n$ odd: $(4,4,1)$\end{tabular} \\[.2em]
\hline
$\Sp(n,\R)$ &$\U(n)$ & $C_n$ &$(0,1,1)$ \\[.2em]
\hline
$\mathfrak{e}_{6(-14)}$ &$\Spin(10)\times\U(1)$ & $BC_2$ &$(8,6,1)$ \\[.2em]
\hline
$\mathfrak{e}_{7(-25)}$ &$\ad(\mathfrak{e}_6)\times \U(1)$ & $C_3$ &$(0,8,1)$ \\[.2em]
\hline
\end{tabular}
\bigskip

\caption{Root multiplicities of irreducible non-compact Hermitian symmetric spaces}
\label{table:Hermitian}
\end{table}

In the following, the triples $m=(m_\rml,m_\rmm,m_\rms)$ appearing as root multiplicities of Riemannian symmetric spaces of the non-compact type and root system of type $BC$ will be called \textit{geometric multiplicities}.

Other examples of geometric multiplicities will be considered in Section \ref{section:geometric}.
\end{Ex}

Notice that in this paper we adopt the notation commonly used in the theory of symmetric spaces. It differs from the notation in the work of Heckman and Opdam in the following ways. The root system $R$ and the multiplicity function $k$ used by Heckman and Opdam
are related to our $\Sigma$ and $m$ by the relations $R=\{2\a:\a \in \Sigma\}$ and $k_{2\a}=m_\a/2$ for $\a \in \Sigma$.

\subsection{Cherednik operators and hypergeometric functions}
\label{subsection:hyp}
In this subsection, we review some basic notions on the hypergeometric functions associated with root systems. This theory has been developed by Heckman, Opdam and Cherednik. We refer the reader to \cite{HS94}, \cite{OpdamActa}, and \cite{Opd00} for more details and further references.

Let $\polya$ denote the symmetric algebra over $\mathfrak{a}_\C$ considered as the
space of polynomial functions on $\frakacs$, and let $\polya^W$ be the
subalgebra of $W$-invariant elements.
Every $p \in \polya$ defines a
constant-coefficient differential operators $p(\partial)$ on
$A_\C$ and on $\mathfrak{a}_\C$
such that $\xi(\partial)=\partial_\xi$ is the directional derivative in the
direction of $\xi$ for all $\xi \in \mathfrak{a}$. The algebra of the differential operators $p(\partial)$
with $p \in \polya$ will also be indicated by $\polya$.

Set $\mathfrak{a}_{\rm reg}=\{x\in \mathfrak{a}: \text{$\alpha(x)\neq 0$ for all $\alpha\in \Sigma$}\}$. 
Let $\mathcal{R}$ denote the algebra of functions on 
$\mathfrak{a}_{\rm reg}$
generated by $1$ and
\begin{equation}
\label{eq:galpha}
g_\a=(1-e^{-2\a})^{-1} \qquad (\a\in \Sigma^+)\,.
\end{equation}
Notice that for $\xi\in \fa$ and $\a\in \Sigma^+$ we have
\begin{eqnarray}
\partial_\xi g_\a&=&2\a(\xi)  \big(g_\a^2-g_\a\big) \notag\\
\label{eq:fminusa}
g_{-\a}&=&1-g_\a\,.
\end{eqnarray}
Hence $\mathcal{R}$ is stable under the actions of  $S(\fa_\C)$ and of the Weyl group.
Let $\D_\mathcal{R}=\mathcal{R}\otimes S(\fa_\C)$ be the algebra of differential operators on $\mathfrak{a}_{\rm reg}$ with coefficients in $\mathcal{R}$. The Weyl group $W$ acts on $\D_\mathcal{R}$
according to
\begin{equation*}
w\big(\phi\otimes p(\partial)):=w\phi \otimes (wp)(\partial).
\end{equation*}
We indicate by $\D_\mathcal{R}^W$ the subalgebra of $W$-invariant elements of $\D_\mathcal{R}$.
The space $\D_\mathcal{R} \otimes \C[W]$ can be naturally endowed with the structure of an associative algebra.

The differential component of a differential-reflection operator $P \in \D_{\mathcal R}\otimes \C[W]$ is the differential operator $\beta(P)\in\D_{\mathcal R}$ such that
\begin{equation}
\label{eq:beta}
Pf=\beta(P)f\,.
\end{equation}
for all $f\in \C[A_\C]^W$. If $P=\sum_{w\in W} P_w\otimes w$ with $P_w\in \D_{\mathcal R}$,
then $\beta(P)=\sum_{w\in W} P_w$.

For $\xi \in \mathfrak{a}$ the Cherednik operator (or Dunkl-Cherednik operator\/)  $T_\xi(m)\in \D_\mathcal{R} \otimes \C[W]$
is defined by
\begin{equation}
\label{eq:T}
  T_\xi(m):=
\partial_\xi-\rho(m)(\xi)+\sum_{\a\in \Sigma^+} m_\a \a(\xi) (1-e^{-2\a})^{-1}
\otimes (1-r_\a)
\end{equation}
%
The Cherednik operators commute with each other (cf.
\cite{OpdamActa}, Section 2). Therefore the map  $\xi \to T_\xi(m)$
on $\mathfrak{a}$ extends uniquely to an algebra homomorphism $p
\to T_p(m)$ of $\polya$ into $\D_\mathcal{R} \otimes \C[W]$.
If $p\in \polya^W$, then $D_p(m):=\beta\big(T_p(m)\big)$ turns out to be in $\D_\mathcal{R}^W$.

Let $p_L \in \polya^W$ be defined by $p_L(\l):=\inner{\l}{\l}$ for $\l \in \frakacs$.
Then $T_{p_L}(m)=\sum_{j=1}^r T_{\xi_j}(m)^2$, where  $\{\xi_j\}_{j=1}^r$ is any orthonormal basis of
$\fa$, is called the Heckman-Opdam Laplacian. Explicitly,
\begin{equation}
\label{eq:HOLaplacian}
T_{p_L}(m)=L(m)+\inner{\rho(m)}{\rho(m)}-\sum_{\a\in \Sigma^+} m_\a \frac{\inner{\a}{\a}}{\sinh^2\alpha} \otimes (1-r_\a)\,,
\end{equation}
where
\begin{equation}
    \label{eq:Laplm}
  L(m):=L_{\mathfrak a}+\sum_{\a \in \Sigma^+} m_\a \,\coth\a \;
       \partial_{x_\a}\,,
\end{equation}
$L_{\mathfrak a}$ is the Laplace operator on $\mathfrak a$, and
$\partial_{x_\a}$ is the directional derivative in the direction
of the element $x_\a\in \fa$  corresponding to $\a$ in the
identification of $\fa$ and $\fa^*$ under $\inner{\cdot}{\cdot}$,
as at the beginning of subsection \ref{subsection:roots}.  See
\cite[(1)]{Sch08} and \cite[\S 2.6]{SchThese} for the computation
of \eqref{eq:HOLaplacian}. In \eqref{eq:HOLaplacian} and
(\ref{eq:Laplm}) we have set 
\begin{eqnarray}
\label{eq:sinh-alpha}
\sinh \a&=&\frac{e^\a-e^{-\a}}{2}=\frac{e^\a}{2}\, (1-e^{-2\a})\,,\\
\label{eq:coth-alpha}
\coth \a&=&\frac{1+e^{-2\a}}{1-e^{-2\a}}=\frac{2}{1-e^{-2\a}}-1\,.
\end{eqnarray}
Moreover,
$$D_{p_L}(m) =L(m)+ \inner{\rho(m)}{\rho(m)}\,.$$

Set $\D(m):=\{D_p(m):p \in \polya^W\}$. The map $\gamma(m): \D(m)\rightarrow \polya^W$ defined by
\begin{equation}
  \label{eq:HChomo}
\gamma(m)\big(D_p(m)\big)(\l):=p(\l)
\end{equation}
is called the Harish-Chandra homomorphism.
It defines an algebra isomorphism of $\D(m)$ onto $\polya^W$ (see \cite[Theorem 1.3.12 and Remark 1.3.14]{HS94}). From Chevalley's theorem it therefore follows that
$\D(m)$ is generated by $r (=\dim \mathfrak{a})$ elements. The next lemma will play a decisive role for us.

\begin{Lemma}
\label{lemma:commutatorL}
For every multiplicity function $m$, the algebra $\D(m)$ is the centralizer of $L(m)$ in $\D_\mathcal{R}^W$.
\end{Lemma}
\begin{proof}
This is \cite[Theorem 1.3.12]{HS94}. See also \cite[Remark 1.3.14]{HS94} for its extension to
complex-valued multiplicities.
\end{proof}

Let $\l \in \frakacs$ be fixed.
The Heckman-Opdam hypergeometric function with spectral parameter $\l$ is the
unique $W$-invariant analytic solution $F_\l(m)$ of the system of differential equations
\begin{equation}
  \label{eq:hypereq1}
D_p(m) f=p(\l)f \qquad (p \in \polya^W),
\end{equation}
which satisfies $f(0)=1$.
The non-symmetric hypergeometric function with spectral parameter $\l$ is the
unique analytic solution $G_\l(m)$ of the system of differential equations
\begin{equation}
  \label{eq:hypereq2}
T_\xi(m) g=\l(\xi) g \qquad (\xi \in \fa),
\end{equation}
which satisfies $g(0)=1$.
These functions are linked by the relation
\begin{equation}
\label{eq:F-G}
F_\l(m;x)=\frac{1}{|W|}\, \sum_{w\in W} G_\l(m;w^{-1}x) \qquad (x \in \fa)\,,
\end{equation}
where $|W|$ denotes the cardinality of $W$.
The existence of the hypergeometric functions requires suitable assumptions on $m$. We refer the reader to Appendix \ref{appendix:A} for additional information.

For geometric multiplicities, $\D(m)$ coincides with the algebra
of radial components of the $G$-invariant differential operators
on $G/K$.
Moreover, $F_\l(m)$ agrees with the restriction to $A\equiv
\fa$  of Harish-Chandra's spherical function on $G/K$ with
spectral parameter $\l$. A geometric intepretation of the
functions $G_\l(m)$ has been recently given in \cite{Oda14}.


\section{Hypergeometric functions associated with standard multiplicities}
\label{section:basic}

In this section we look at the hypergeometric functions associated with standard multiplicites. Basic estimates and a characterization of the bounded hypergeometric functions (for these sets of multiplicity functions) are established.

Let $\mathcal{M}$ denote the set of real-valued multiplicity functions $m=(m_\rms,m_\rmm,m_\rml)$ on a root system $\Sigma$ of type $BC_r$.
We will consider the following subsets of $\mathcal{M}$:
\begin{eqnarray}
\label{eq:M+}
\mathcal{M}_+&=&\{m\in \mathcal{M}:
\text{$m_\a\geq 0$ for every $\alpha\in\Sigma$}\}\,,\\
\label{eq:M0}
\mathcal{M}_0&=&\{(m_\rms,m_\rmm,m_\rml)\in \mathcal{M}: m_\rmm\geq 0, m_\rms+m_\rml\geq 0\}\,,\\
\label{eq:M1}
\mathcal{M}_1&=&\{(m_\rms,m_\rmm,m_\rml)\in \mathcal{M}:  m_\rmm> 0, m_\rms> 0, m_\rms+2m_\rml> 0\}\,,\\
\label{eq:M2}
\mathcal{M}_2&=&\{(m_\rms,m_\rmm,m_\rml)\in \mathcal{M}: m_\rmm\geq 0, m_\rml\geq 0, m_\rms+m_\rml\geq 0\}\,,\\
\label{eq:M3}
\mathcal{M}_3&=&\{(m_\rms,m_\rmm,m_\rml)\in \mathcal{M}: m_\rmm\geq 0, m_\rml\leq 0, m_\rms +2m_\rml\geq 0\}\,.
\end{eqnarray}
So $\mathcal{M}_+$ consists of the non-negative multiplicity functions
and $\mathcal{M}_1=(\mathcal{M}_+\cup \mathcal{M}_3)^0$, the interior of
$\mathcal{M}_+\cup \mathcal{M}_3$.
For real-valued multiplicities, $\mathcal{M}_0$ is the natural set for which both hypergeometric functions $G_\lambda(m)$ and $F_\lambda(m)$ are defined for all $\lambda\in\mathfrak{a}^*_\C$; see Appendix A.
Recall that the elements of $\mathcal{M}_1$ are called standard multiplicity functions (see \cite[Definition 5.5.1]{HS94}). These sets of multiplicities are represented in Figure \ref{fig:m}.
\begin{center}
\begin{figure}[H]
\begin{tikzpicture}[
    scale=1.8,
    axis/.style={very thick, ->, >=stealth'},
    equation line/.style={thin},
    ]
   \fill[gray!20!,path fading=east, fading angle=45]
    (0,0) -- (1.8,0) -- (1.8,1.8) -- (0,1.8) -- cycle;
   \fill[gray!110!,path fading=east]
   (0,0) -- (1.8,0) -- (1.8,-.9) -- cycle;
   \fill[gray!40!,path fading=east]
   (0,0) -- (1.8,-.9) -- (1.4,-1.4) -- cycle;
   \fill[gray!40!,path fading=north]
   (0,0) -- (0,1.8) -- (-1.4,1.4) -- cycle;
    \draw[axis] (-1.4,0)  -- (2,0) node(xline)[right]
        {$m_{\mathrm{s}}$};
    \draw[axis] (0,-1.4) -- (0,2) node(yline)[above]
        {$m_{\mathrm{l}}$};
     \draw[equation line] (0,0) -- (1.8,-.9)
        node[right, text width=10em, rotate=0]
        {$m_{\mathrm{s}}+2m_{\mathrm{l}}=0$};
     \draw[equation line] (-1.4,1.4) -- (1.4,-1.4)
        node[right, text width=10em, rotate=0]
        {$m_{\mathrm{s}}+m_{\mathrm{l}}=0$};
     \draw [->,line width=.5pt] (.8,-.4)
        arc[x radius=.9cm, y radius =.9cm,
        start angle=-22.5, end angle=88];
     \draw [->,line width=.5pt] (-.5,.5)
        arc[x radius=.7cm, y radius =.7cm,
        start angle=135, end angle=0];
    \draw (1.2,1.5) node[right] {$\mathcal{M}_+$};
  	\draw (-.67,.78) node[right] {{\footnotesize$\mathcal{M}_2$}};
  	\draw (.8,.5) node[right] {{\footnotesize $\mathcal{M}_1$}};
  	\draw (1.2,-.4) node[right] {$\mathcal{M}_3$};
\end{tikzpicture}
\caption{Sets of $BC$ multiplicities}
\label{fig:m}
\end{figure}
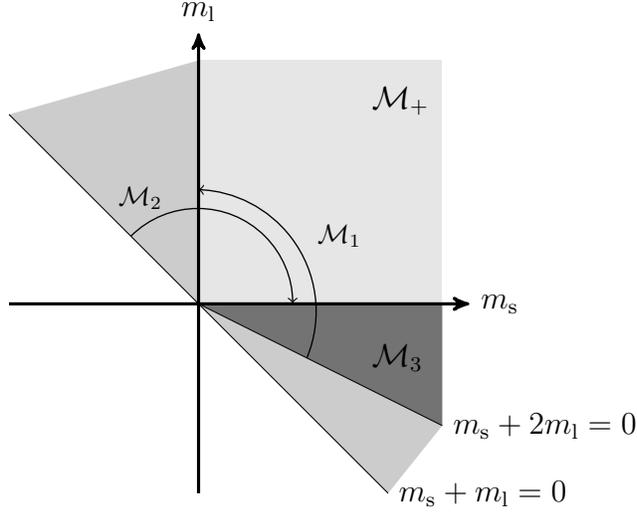
\end{center}


\subsection{Basic estimates}
\label{subsection:estimates}

Under the assumption that $m\in \mathcal{M}_+$, the positivity properties of  $F_{\l}(m)$ and $G_{\l}(m)$ for $\l\in\fa^*$ as well as their basic estimates have been proved by Schapira in \cite[\S 3.1]{Sch08}, by refining some ideas from \cite[\S 6]{OpdamActa}. Under the same assumption, Schapira's estimates for $F_\l(m)$ and $G_\l(m)$ have been sharpened by R\"osler, Koornwinder and Voit in \cite[\S 3]{RKV13}.
The following proposition collects their results.

\begin{Prop}
\label{prop:positivity}
Let $m\in \mathcal{M}_+$.
Then the following properties hold.
\begin{enumerate}
\thmlist
\item For all $\l\in\fa^*$ the functions $G_\l(m)$ and  $F_\l(m)$ are real and strictly positive.
\item For all $\l\in\fa_\C^*$
\begin{equation}
\label{eq:basic-estimate1}
\begin{split}
|G_\l(m)|&\leq G_{\Re\l}(m)\,,\\
|F_\l(m)|&\leq F_{\Re\l}(m)\,.
\end{split}
\end{equation}
\item For all $\l\in\fa^*$ and all $x\in \fa$
\begin{equation}
\begin{split}
\label{eq:basic-estimate2}
G_\l(m;x)&\leq G_0(m;x)e^{\max_w (w\l)(x)} \\
F_\l(m;x)&\leq F_0(m;x)e^{\max_w (w\l)(x)} \,.
\end{split}
\end{equation}
More generally, for all $\l\in\fa^*$, $\mu\in \overline{(\fa^*)^+}$ and all $x\in \fa$
\begin{equation}
\begin{split}
\label{eq:basic-estimate3}
G_{\l+\mu}(m;x)&\leq G_\mu(m;x)e^{\max_w (w\l)(x)} \\
F_{\l+\mu}(m;x)&\leq F_\mu(m;x)e^{\max_w  (w\l)(x)} \,.
\end{split}
\end{equation}
\end{enumerate}
(In the above estimates, $\max_w$ denotes the maximum over all $w\in W$.)
\end{Prop}

Proposition \ref{prop:positivity} lists the sharpest estimates known so far for Heckman-Opdam’s hypergeometric functions. The following lemma, even if quite elementary, is the key result allowing us to modify their proofs and to extend these estimates to multiplicities $m\in \mathcal{M}_3$.. It will play a similar role also for extending the asymptotics and the boundedness characterization..

\begin{Lemma}
\label{lemma:pos-coeffs}
Let $m=(m_\rms,m_\rmm,m_\rml)\in \mathcal{M}_0$ and $\beta\in \Sigma_{\rml}^+$. Then  the following inequalities hold for all $x\in \mathfrak{a}$:
\begin{enumerate}
\thmlist
\item
$\displaystyle{\frac{m_\rms}{2}+m_\rml \, \frac{1}{1+e^{-\beta(x)}}\geq 0}$\quad  if $m\in \mathcal{M_+}\cup \mathcal{M}_3$;
\smallskip

\item
$\displaystyle{\frac{m_\rms}{2}+m_\rml \, \frac{1+e^{-2\beta(x)}}{(1+e^{-\beta(x)})^2}\geq 0}$\quad  if $m\in \mathcal{M}_2\cup \mathcal{M}_3$.
\end{enumerate}
\end{Lemma}
\begin{proof}
If $m\in \mathcal{M_+}\cup \mathcal{M}_3$, then $m_\rms\geq 0$ and $m_\rml\geq -\frac{1}{2}m_\rms$. Hence, for every $0\leq C\leq 1$, we have
$\frac{m_\rms}{2}+C m_\rml  \geq (1-C)\frac{m_\rms}{2} \geq 0$.
We therefore obtain both (a) and (b) for these $m$'s by observing that for every $t=-\b(x)\in \R$ we have
$$0\leq \frac{1}{1+e^t}\leq 1 \qquad \text{and} \qquad 0\leq\frac{1+e^{2t}}{(1+e^t)^2}\leq 1\,.$$

Suppose now that $m\in \mathcal{M}_2$. Since $m_\rml\geq 0$, $m_\rms+m_\rml\geq 0$ and $\frac{1+|z|^2}{|1+z|^2}\geq \frac{1}{2}$ for all $z\in \C$, we immediately
obtain that for all $t=-\beta(x)\in \R$ we have
$$\frac{m_\rms}{2}+m_\rml \, \frac{1+e^{2t}}{(1+e^t)^2} \geq \frac{1}{2}(m_\rms+m_\rml) \geq 0\,.$$
This proves the lemma.
\end{proof}

\begin{Rem}
Since
$$
\lim_{t\to+\infty} \Big(\frac{m_\rms}{2}+m_\rml \frac{1}{1+e^t}\Big)=\frac{m_\rms}{2}
\quad \text{and} \quad
\left.\Big( \frac{1}{2} m_\rms+m_\rml\frac{1+e^{2t}}{(1+e^t)^2} \Big)\right|_{t=0}=\frac{1}{2}m_\rms+m_\rml\,,$$
it is clear that the inequality (a) of Lemma \ref{lemma:pos-coeffs} cannot extend to $m\in
\mathcal{M}_0 \setminus \big(\mathcal{M}_+\cup \mathcal{M}_3\big)$ and (b) cannot extend to $m\in
\mathcal{M}_0 \setminus \big(\mathcal{M}_2\cup \mathcal{M}_3\big)$.
\end{Rem}

To use the methods from \cite[\S 3]{RKV13}, we will
also need an extension of the real case of  Opdam's estimates, \cite[Proposition 6.1 (1)]{OpdamActa}. We will do this for $m\in \mathcal{M}_2\cup \mathcal{M}_3$.
Notice that the complex variable version of Opdam's estimates, \cite[Proposition 6.1 (2)]{OpdamActa}, has been recently extended by Ho and \'Olafsson \cite[Appendix A]{HoOl14} to $m\in \mathcal{M}_2$ and to a domain in $\mathfrak{a}_\C$ which is much larger than the original one considered by Opdam. For $m\in \mathcal{M}_2$  the inequality (b) of Lemma \ref{lemma:pos-coeffs} was noticed in the proof of \cite[Proposition 5.A, p. 25]{HoOl14}.

\begin{Lemma}
\label{lemma:OpdamM2M3}
Suppose that $m\in \mathcal{M}_2 \cup \mathcal{M}_3$.
Then for all $\l\in\fa_\C^*$ and $x\in \fa$
\begin{equation}
\label{eq:basic-estimate4}
\begin{split}
|G_\l(m;x)|&\leq \sqrt{|W|} e^{\max_w \Re(w\l)(x)}\,,\\
|F_\l(m;x)|&\leq \sqrt{|W|} e^{\max_w \Re(w\l)(x)}\,.
\end{split}
\end{equation}
\end{Lemma}
\begin{proof}
Regrouping the terms corresponding to
$\b$ and $\b/2$ and setting $z=x\in \mathfrak{a}$, we can
rewrite the first term on the right-hand side of the first displayed equation in \cite[p. 101]{OpdamActa} as
\begin{eqnarray*}
&&\sum_{w\in W, \a\in \Sigma_\rmm^+}
m_\a \frac{\a(\xi)(1-e^{-4\a(x)})}{(1-e^{-2\a(x)})^2} |\phi_w-\phi_{r_\a w}|^2 e^{-2\mu(x)}+\\
&&\quad \sum_{w\in W, \beta\in \Sigma_\rml^+}
\Big(m_{\beta/2} \frac{\frac{\b}{2}(\xi)(1-e^{-2\beta(x)})}{(1-e^{-\beta(x)})^2} +
m_{\beta} \frac{\beta(\xi)(1-e^{-4\beta(x)})}{|1-e^{-2\beta(x)})^2} \Big) |\phi_w-\phi_{r_\beta w}|^2 e^{-2\mu(x)}
\end{eqnarray*}
in which
$\alpha(\xi) (1-e^{-4\a(x)})$ is positive for all $\a\in \Sigma^+$ when $\xi\in \mathfrak{a}$ is chosen in the same Weyl chamber of $x$. 
(In \cite[p. 101]{OpdamActa}, there is an additional multiplicative constant $-\frac{1}{2}$. Our argument takes into account this sign change.)
The coefficient of $ |\phi_w-\phi_{r_\beta w}|^2 e^{-2\mu(x)}$ for $\beta\in \Sigma_\rmm^+$ is then clearly positive. For $\beta\in \Sigma_\rml^+$, this coefficient is equal to
$$
\frac{\beta(\xi)(1-e^{-2\beta(x)})}{(1-e^{-\beta(x)})^2}
\Big(\frac{m_\rms}{2} +m_\rml
\frac{1+e^{-2\beta(x)}}{(1+e^{-\beta(x)})^2} \Big)\,,$$
which is positive by Lemma \ref{lemma:pos-coeffs}(b).
Therefore, the same proof as in \cite[Proposition 6.1]{OpdamActa} allows us to obtain the required inequalities.
\end{proof}

The following proposition extends the positivity properties and basic estimates from \cite[\S 3.1]{Sch08} and \cite[\S 3]{RKV13} to the multiplicity functions in $\mathcal{M}_3$.

\begin{Prop}
\label{prop:positivity-estimates-M3}
Proposition \ref{prop:positivity} holds for $m\in \mathcal{M}_3$.
\end{Prop}
\begin{proof}
The proofs of (a), (b) and \eqref{eq:basic-estimate2} follow the same steps as the proofs of \cite[Lemma 3.1 and Proposition
3.1]{Sch08}. So we shall just indicate what has to be
modified in the proofs when considering $m\in \mathcal{M}_3$ instead of $m\in \mathcal{M}_+$.

Suppose $G_\l(m)$ is not positive and let $x\in \fa$ be a zero of $G_\l(m)$ of minimal norm.
As in \cite[Lemma 3.1]{Sch08}, one has to distinguish whether $x$ is regular or singular. If $x$ is regular, take $\xi$ in the same chamber of $x$.
Evaluation at $x$ of the equation
\begin{equation}
\label{eq:hyperGlambda}
T_\xi(m) G_\l(m)=\l(\xi) G_\l(m)
\end{equation}
yields
\begin{multline*}
\partial_\xi G_\l(m;x)=\sum_{\a\in \Sigma^+} m_\alpha \frac{\a(\xi)}{1-e^{-2\a(x)}}
\big[G_\l(m;r_\a x)-G_\l(m;x)\big]
+\big(\rho(m)+\l\big)(\xi) G_\l(m;x)
\end{multline*}
in which $G_\lambda(m;x)$ vanishes.
In the sum over $\Sigma^+$, the coefficient of
$G_\l(m;r_\a x)-G_\l(m;x)$ is always non-negative for
$\a\in \Sigma_\rmm^+$. Moreover, grouping together
those corresponding to $\b/2$ and $\b\in \Sigma_\rml^+$, we obtain
as coefficient of
$G_\l(m;r_\b x)-G_\l(m;x)$
\begin{equation}
\label{eq:positivity-beta}
m_{\beta/2} \frac{\frac{\beta}{2}(\xi)}{1-e^{-\beta(x)}}
+m_{\beta} \frac{\beta(\xi)}{1-e^{-2\beta(x)}}
=\frac{\b(\xi)}{1-e^{-\b(x)}}\Big[ \frac{m_\rms}{2} +
 m_\rml \frac{1}{1+e^{-\b(x)}}\Big],
\end{equation}
which is positive by Lemma \ref{lemma:pos-coeffs}(a).

If $x$ is singular, let $I=\{\a \in \Sigma^+: \a(x)=0\}$ and let $\xi$ be in the same face of $x$ so that $\alpha(\xi)=0$ for all $\alpha\in I$. In this case,
\eqref{eq:hyperGlambda} evaluated at $x$ gives
\begin{multline*}
\partial_\xi G_\l(m;x)=-\sum_{\a\in I} m_\a \frac{\a(\xi)}{\inner{\a}{\a}}\partial_{x_\a}  G_\l(m;x)\\
+\sum_{\alpha\in \Sigma_{\rmm}^+\setminus I}
m_{\rmm}  \frac{\a(\xi)}{1-e^{-2\a(x)}}
\big[G_\l(m;r_\a x)-G_\l(m;x)\big]\\
+\sum_{\b\in\Sigma_{\rml}^+ \setminus I}
\Big[m_{\beta/2} \frac{\frac{\beta}{2}(\xi)}{1-e^{-\beta(x)}}
+m_{\beta} \frac{\beta(\xi)}{1-e^{-2\beta(x)}}\Big]
\big[G_\l(m;r_\beta x)-G_\l(m;x)\big]\\
+\big(\rho(m)+\l\big)(\xi) G_\l(m;x)\,,
\end{multline*}
in which the first sum and $G_\l(m;x)$ vanish, and
the sum over $\Sigma^+_\rml\setminus I$ has coefficient as in \eqref{eq:positivity-beta}.

In both cases, one can argue as in
\cite[Lemma 3.1]{Sch08} by replacing the multiplicities $k_\a\geq 0$ in that proof with
$\frac{m_\rms}{2} +
 m_\rml\, \frac{1}{1+e^{-\b(x)}}$ and using Lemma \ref{lemma:pos-coeffs}(a).

Having that $G_{\Re \l}(m)$ is real and positive, to prove (b) for $G_\l(m)$, one can make the same grouping and substitution of multiplicities, as done above for $\partial_\xi G_\l(m;x)$, inside the formula for $\partial_\xi |Q_\l|^2(x)$ appearing in the proof of \cite[Proposition 3.1 (a)]{Sch08}. A third application of the same grouping in the formula of $\partial_\xi R_\l(x)$ in the proof of \cite[Proposition 3.1 (b)]{Sch08} yields \eqref{eq:basic-estimate2}.

Finally, \eqref{eq:basic-estimate3} has been proven for
multiplicities $m_\a\geq 0$ in \cite[Theorem 3.3]{RKV13} using the
original versions of (a), (b), (c) and  \eqref{eq:basic-estimate2}
together with a clever application of the Phragm\'en-Lindel\"of
principle that does not involve the root multiplicities. With Opdam's estimate for $m\in \mathcal{M}_3$, as in Lemma \ref{lemma:OpdamM2M3}, the original proof from \cite{RKV13}
extends to the case of $m\in \mathcal{M}_3$ and yields \eqref{eq:basic-estimate3}.
\end{proof}

\subsection{Asymptotics}
\label{subsection:asymptotics}

We now investigate the asymptotic behavior of
the hypergeometric functions $F_\l(m)$ for $m\in \mathcal{M}_{1}$.
For this we follow \cite{NPP}. Before we state our results, it is useful to recall
the methods adopted in \cite{NPP}. Let $m$ be an arbitrary
non-negative multiplicity function. One of the main results in
\cite{NPP} is Theorem 2.11, where a series expansion away from the
walls was obtained for $F_\lambda(m) $ for all $\lambda \in
\frakacs$ (even when $\lambda$ is non-generic). Recall from \cite[\S 4.2]{HS94} (or
from \cite[\S 1.2]{NPP} in symmetric space notation) that, for
generic $\lambda$, the function $F_\lambda(m)$ is given on $\fa^+$ by
$\sum_{w \in W} c(m;w\lambda) \Phi_{w\lambda}(m)$, where
$c(m;\lambda)$ is Harish-Chandra's $c$-function (see \eqref{eq:c} in the appendix)  and
$\Phi_\lambda(m;x)$ admits the series expansion
$$\Phi_\lambda(m;x) = e^{(\lambda-\rho(m))(x)} \sum_{\mu \in 2\Lambda}
\Gamma_\mu(m;\lambda) e^{-\mu(x)} \qquad (x\in \fa^+).$$
The coefficients
$\Gamma_\mu(m,\lambda)$ are determined from the recursion relations
$$ \inner{\mu}{\mu-2\lambda} \Gamma_\mu(m,\lambda)= 2 \sum_{\alpha \in
\Sigma^+} m_\alpha \sum_{n \in \mathbb N, \mu-2n\alpha \in
\Lambda} \Gamma_{\mu-2n\alpha}(m,\lambda) \inner
{\mu+\rho(m)-2n\alpha-\lambda}{\alpha} ,$$ with the initial
condition that $\Gamma_0(m, \lambda)= 1$.
The above defines
$\Gamma_\mu(m;\lambda)$ as meromorphic functions on $\frakacs.$

For a non-generic point $\lambda = \lambda_0$ a series expansion for
$F_\lambda(m)$  was obtained in \cite{NPP} following the steps given
below:

\medskip
\textit{Step I:} List the possible singularities of the
$c$-function and the coefficients  $\Gamma_\mu(m;\lambda)$ at $\lambda =
\lambda_0.$

\smallskip
\textit{Step II:} Identify a polynomial $p$ so that $$\lambda \to
p(\lambda) \Big ( \sum_{w \in W}
c(m;w\lambda)e^{(\lambda-\rho(m))(x)} \sum_{\mu \in 2\Lambda}
\Gamma_\mu(m;\lambda) e^{-\mu(x)} \Big )$$ is holomorphic in a
neighborhood of $\lambda_0.$

\smallskip
\textit{Step III:} Write $F_{\lambda_0}(m) = a~\partial(\pi)
(pF_\lambda(m))|_{\lambda = \lambda_0}$ where $\partial(\pi)$ is the
differential operator corresponding to the highest degree
homogenous term in $p$ and $a,$ is a non-zero constant. This gives
the series expansion of $F_{\lambda_0}(m).$

\medskip
A careful examination shows that the same proofs go through even
with the assumption that the multiplicity function belongs to
$\mathcal{M}_{1}$.
Indeed, the possible singularities of
the $c$-function and the $\Gamma_\mu(m)$ are contained in the same set of
hyperplanes as listed in \cite[Lemma 2.3]{NPP}. Hence the same
polynomials and differential operators can be used in Step II and
III above. The crucial detail to be checked is the
computation of the constant $b_0(m;\lambda_0)$ appearing in
\cite[Lemma 2.6]{NPP}.
The explicit expression in terms of
Harish-Chandra's $c$-function shows that, for root systems of type $BC_r\setminus C_r$, the function $b_0(m;\lambda_0)$ is
non-zero if and only if
\begin{equation}
\label{eq:b-nonzero}
\prod_{\alpha \in \Sigma^+_\rms}
\Gamma\Big(\frac{(\lambda_0)_\alpha}{2} +\frac{m_\rms}{4}+\frac{1}{2}\Big)
\Gamma\Big(\frac{(\lambda_0)_\alpha}{2} +\frac{m_\rms}{4}+\frac{m_\rml}{2}\Big)
\prod_{\alpha \in \Sigma^+_\rmm}
\Gamma\Big(\frac{(\lambda_0)_\alpha}{2} +\frac{m_\rmm}{4}+\frac{1}{2}\Big) \Gamma\Big(\frac{(\lambda_0)_\alpha}{2} +\frac{m_\rmm}{4}\Big)
\end{equation}
(where the second factor in the product does not appear if the rank $r$ is one)
is nonsingular. 
It is clear that the expression in \eqref{eq:b-nonzero} is nonsingular 
for $\lambda_0 \in \frakacs$ with $\Re
\lambda_0 \in \overline{(\fa^\ast)^+}$ if $m\in \mathcal{M}_{1}$.
Notice that the nonvanishing of $b_0(m;\lambda_0)$ identifies the main term in the expansion of
$F_{\lambda_0}(m;x)$ as
$\frac{b_0(m;\lambda_0)}{\pi_0(\rho_0(m))} \pi_0(x) e^{(\lambda_0 - \rho(m))(x)}$,
where $\rho_0(m)$ is defined as in \cite[(58)]{NPP}.
Notice also that $c(m;\lambda_0)$, and hence $b_0(m;\lambda_0)$, can vanish
for $m\in \mathcal{M}_0\setminus \mathcal{M}_1$.

It follows from the above that \cite[Theorem 2.11]{NPP} and, as a
consequence, \cite[Theorem 3.1]{NPP} continue to hold true for a
multiplicity function $m\in \mathcal{M}_{1}$.
We state it below and refer to \cite{NPP} for any unexplained notation.

\begin{Thm}
\label{thm:hc-expansion}
Suppose $m\in \mathcal{M}_{1}$. Let $\lambda_0 \in \frakacs$ with $\Re
\lambda_0 \in \overline{(\fa^\ast)^+},$ and let $x_0 \in \fa^+$ be
fixed. Then, there are constants $C_1 > 0, C_2 > 0$ and $b > 0$
(depending on $m$, $\lambda_0$ and $x_0$) so that for all $x \in x_0 +
\fa^+ :$
\begin{multline} \label{eq:restF-est}
\Big| \frac{F_{\l_0}(m; x)
e^{-(\Re\l_0-\rho(m))(x)}}{\pi_0(x)} -\Big(
\frac{b_0(m;\l_0)}{\pi_0(\rho_0(m))} e^{i\Im\l_0(x)} + \sum_{w\in
W_{\Re\l_0} \setminus W_{\l_0}} \! \!
\frac{b_w(m;\l_0)\pi_{w,\l_0}(x)}{c_0\pi_0(x)}
e^{i w\Im\l_0(x)} \Big)\Big|  \\
\leq C_1 (1+\b(x))^{-1} +C_2 (1+ \b(x))^{|\Sigma_{\l_0}^+|}
e^{-b\beta(x)}\,,
\end{multline}
where $\beta(x)$ is the minimum of $\alpha(x)$ over the simple roots $\alpha\in \Sigma^+$
and the term $C_1 (1+\b(x))^{-1}$ on the right-hand side of
(\ref{eq:restF-est}) does not occur if $\inner{\alpha}{\l_0}\neq 0$ for all $\alpha\in \Sigma$.
\end{Thm}

Notice that, for fixed $x_0\in \fa^+$, we have $\beta(x) \asymp |x|$ as $x\to \infty$ in $ x_0+\overline{\fa^+}$, where $|x|$ is the Euclidean norm on $\fa$.

Likewise, one can extend to $m\in \mathcal{M}_{1}$ the following corollary, which restates Theorem \ref{thm:hc-expansion} in the special case where $\l_0\in \overline{(\mathfrak{a}^*)^+}$.

\begin{Cor} \label{cor:leading-termF-realcase}
Suppose $m\in \mathcal{M}_{1}$.  Let $\l_0 \in \overline{(\mathfrak{a}^*)^+}$, and let 
$x_0 \in \mathfrak{a}^+$ be fixed. Then there are constants $C_1>0$, $C_2>0$ and $b>0$
(depending on $m$, $\l_0$ and $x_0$) so that for all $x \in x_0+
\overline{\mathfrak{a}^+}$:
\begin{multline} \label{eq:restF-est-realcase}
\Big| F_{\l_0}(m; x) -\frac{b_0(m;\l_0)}{\pi_0(m;\rho_0)} \pi_0(x) e^{(\l_0-\rho(m))(x)}\Big|  \leq \\
\leq \big[ C_1 (1+\b(x))^{-1} +C_2 (1+ \b(x))^{|\Sigma_{\l_0}^+|}
e^{-b\beta(x)}\big] \pi_0(x) e^{(\l_0-\rho(m))(x)}\,.
\end{multline}
The term $C_1 (1+\b(x))^{-1}$ on the right-hand side of
(\ref{eq:restF-est-realcase}) does not occur if $\inner{\alpha}{\l_0}\neq 0$ for all $\alpha\in \Sigma$.
\end{Cor}

It might be useful to observe that for $m\in \mathcal{M}_0$ the only obstruction to \eqref{eq:restF-est} is that $b_0(m;\lambda_0)\neq 0.$ For arbitrary values of $\lambda_0 \in \frakacs$ with $\Re
\lambda_0 \in \overline{(\fa^\ast)^+}$, by \eqref{eq:b-nonzero}, this happens if and only if
$$
m_\rms>-2\,, \qquad m_\rms+2 m_\rml >0 \quad \text{and, if $r>1$} \quad m_\rmm>0\,.
$$

\begin{Cor}
\label{cor:asymptotics-b0}
Let $m\in \mathcal{M}_0$. Then Theorem \ref{thm:hc-expansion} holds for every $\l_0\in \mathfrak{a}_\C^*$ with $\Re\l_0\in\overline{(\mathfrak{a}^*)^+}$ such that
\eqref{eq:b-nonzero} is nonsingular.
\end{Cor}

\subsection{Sharp estimates}
\label{subsection:sharp}

In this subsection we assume that $m\in \mathcal{M}_3$.
Let $\mathcal{M}_3^0$ denote the interior of $\mathcal{M}_3$.
Since $\mathcal{M}_3^0 \subset \mathcal{M}_1$, the results of both subsections \ref{subsection:estimates} and \ref{subsection:asymptotics} are available on $\mathcal{M}_3^0$.

Let $\l\in \mathfrak{a}^*$. Using the nonsymmetric hypergeometric
functions $G_\l(m)$, their relation to the hypergeometric
function, and the system of differential-reflection equations they satisfy, Schapira proved in \cite{Sch08} the following local Harnack principle for the hypergeometric function $F_{\l}(m)$: for all $x \in \overline{\mathfrak{a}^+}$
\begin{equation} \label{eq:localHarnack}
\nabla F_{\l}(m;x)=-\frac{1}{|W|} \sum_{w \in W}
w^{-1}(\rho(m)-\l)G_{\l}(m;wx)\,,
\end{equation}
the gradient being taken with respect to the space variable $x\in
\mathfrak{a}$. It holds for every multiplicity function $m$ for which both $G_\l(m)$ and $F_\l(m)$ are defined. See \cite[Lemma 3.4]{Sch08}. Since
$\partial_\xi F=\inner{\nabla F}{\xi}$ and since $G_\l(m)$ and $F_\l(m)$ are real and non-negative for $m\in \mathcal{M}_3$, one obtains as in \textit{loc. cit.}  that for all $\xi \in \mathfrak{a}$
$$
\partial_\xi \Big( e^{K_\xi \frac{\inner{\xi}{\cdot}}{|\xi|^2}} F_{\l}(m;\cdot) \Big) \geq 0\,,
$$
where $K_\xi=\max_{w\in W} (\rho(m)-\l)(w\xi)$. See Appendix \ref{appendix:B} for a proof.
This in turn yields
the following subadditivity property, which is implicit in
\cite{Sch08} for $m\in \mathcal{M}_+$ and in fact holds also for $m\in \mathcal{M}_3^0$ and, by continuity, on $\mathcal{M}_3$.

\begin{Lemma}\label{lemma:subadd-Schapira}
Suppose $m\in \mathcal{M}_3$.
Let $\l\in\mathfrak{a}^*$. Then for all $x, x_1 \in \mathfrak{a}$ we have
\begin{equation} \label{eq:subadd-Schapira}
F_{\l}(m; x+x_1) e^{\min_{w\in W} (\rho(m)-\l)(wx_1)}
\leq F_\l(m; x) \leq F_{\l}(m; x+x_1) e^{\max_{w\in W}
(\rho(m)-\l)(wx_1)}\,.
\end{equation}
In particular, if $\l \in \overline{(\mathfrak{a}^*)^+}$ and $x_1 \in
\mathfrak{a}^+$, then
\begin{equation} \label{eq:subadd-Schapira-special}
F_{\l}(m; x+x_1) e^{-(\rho(m)-\l)(x_1)} \leq
F_{\l}(m; x) \leq F_{\l}(m; x+x_1)
e^{(\rho(m)-\l)(x_1)}\,
\end{equation}
for all $x \in \mathfrak{a}$.
\end{Lemma}

Together with Corollary \ref{cor:leading-termF-realcase}, the
above lemma yields the following global estimates of
$F_{\l}(m; x).$

\begin{Thm}\label{thm:real}
Let $m\in \mathcal{M}_3^0 \subset \mathcal{M}_1$ and $\l_0 \in \overline{(\mathfrak{a}^*)^+}$. Then for all $x \in
\overline{\mathfrak{a}^+}$ we have
\begin{equation}
F_{\l_0}(m; x) \asymp \big[\prod_{\a\in\Sigma_{\l_0}^0}
(1+\a(x))\big] e^{(\l_0-\rho(m))(x)}\,,
\end{equation}
where $\Sigma_{\l_0}^0=\{ \a\in \Sigma^+_\rms\cup \Sigma^+_\rmm:\inner{\a}{\l_0}=0\}$.
\end{Thm}
\begin{proof}
Same as in \cite[Theorem 3.4]{NPP}.
\end{proof}

We end this section with the following characterization of the bounded hypergeometric functions corresponding to multiplicity functions $m \in \mathcal{M}_1.$

\begin{Thm}\label{thm:bddhyp}
Let $m \in \mathcal{M}_1.$ Then $F_\lambda(m)$ is bounded if and only if $\lambda \in C(\rho(m)) + i \fa^\ast,$ where $C(\rho(m))$ is the convex hull of
the set $\{w\rho(m):~w \in W\}.$
\end{Thm}

\begin{proof}
The same arguments as in \cite[Theorem 4.2]{NPP} work using the results in this section.
\end{proof}


\section{Applications and developments}
\label{section:applications}
In this section we consider two-parameter deformations of the multiplicities in $\mathcal{M}_+$ and study a class of hypergeometric functions associated with them. 
As we shall see in Section \ref{section:geometric}, for specific values of $m\in\mathcal{M}_+$ and of the deformation parameters 
$(\ell,\well)$, these hypergeometric functions turn out to agree with the 
$\tau$-spherical functions on the homogeneous vector bundles over $G/K$, when $\tau$ is a small $K$-type and $G/K$ has root system of type $BC$. 
The general properties proved in this section will provide for  
most of the $BC$ cases in \cite{OS17} symmetry properties, estimates, aymptotics and a characterization of the $\tau$-spherical functions which are bounded.

\subsection{A two-parameter deformation of a multiplicity function}
\label{subsection:deformation}

Let $(\fa,\Sigma,m)$ be a triple as in subsection \ref{subsection:roots}, with $m=(m_\rms,m_\rmm,m_\rml)$. For any two parameters $\ell,\well $ we define a deformation 
$m(\ell,\well)$ of $m$ as follows:
\begin{equation}
\label{eq:mult-l-tl-gen}
m_\a(\ell,\well)=\begin{cases}
m_\rms+2\ell  &\text{if $\a\in \Sigma_\rms$}\\
m_\rmm +2\well &\text{if $\a\in \Sigma_\rmm$}\\
m_\rml-2\ell  &\text{if $\a\in \Sigma_\rml$}\,.
\end{cases}
\end{equation}
We shall suppose in the following that $m\in \mathcal{M}_0$, $\ell, \well\in \R$ and 
$\well\geq -m_\rmm/2$. 
When this does not cause any confusion, we shall shorten the notations $m(\ell,0)$ and $m(0,\well)$ and write $m(\ell)$ and $m(\well)$, respectively. Since $m_\rms(\ell, \well)+m_\rml(\ell, \well)=m_\rms+m_\rml$, the above assumptions ensure that $m(\ell, \well)\in \mathcal{M}_0$.
The two deformations, in $\ell$ and $\well$, are independent. So, $m(\ell,\well)=m(\ell)(\well)=
m(\well)(\ell)$.
Since we are assuming that $\well\geq -m_\rmm/2$, the additional parameter $\well$
does not increase the range of 
possible multiplicities of the middle roots.
Its relevance will appear in the definitions the of $\tauwell$-hypergeometric functions in \eqref{eq:Fell-F} and \eqref{eq:Gell-G}. 

The following theorem shows that every element of
$\mathcal{M}_+ \cup \mathcal{M}_3$ is of the form $m(\ell)$ for some $m\in \mathcal{M}_+$. 
For a fixed $m=(m_\rms,m_\rmm,m_\rml)$ we shall use the notation
\begin{equation}
\label{eq:lmin-lmax}
\ell_{\rm min}(m)=-\frac{m_\rms}{2}\qquad \text{and}\qquad  \ell_{\rm max}(m)=\frac{m_\rms}{2}+m_\rml\,.
\end{equation}
We simply write $\ell_{\rm min}$ and $\ell_{\rm max}$
when this does not cause any ambiguity.

\begin{Lemma}
\label{lemma:M+M3}
Let $m^0=(m^0_\rms, m^0_\rmm,m^0_\rml)\in \mathcal{M}_0$. Then  $m^0\in \mathcal{M}_+ \cup \mathcal{M}_3$ if and only if there are
$m=(m_\rms,m_\rmm,m_\rml)\in \mathcal{M}_+$ and
$\ell\in \R$ so that $m^0=m(\ell)$ and
$\ell \in \big[\ell_{\rm min}(m),\ell_{\rm max}(m)\big]$.
Moreover, $m^0\in \mathcal{M}_1=(\mathcal{M}_+ \cup \mathcal{M}_3)^0$ if and only if  $m=(m_\rms,m_\rmm,m_\rml)\in \mathcal{M}_+$ is as above and
$\ell \in \big]\ell_{\rm min}(m),\ell_{\rm max}(m)\big[$.
\end{Lemma}
\begin{proof}
If $m^0\in \mathcal{M}_+ \cup \mathcal{M}_3$, then
$m^0=m(\ell)$ for
$$
m=(m_\rms,m_\rmm,m_\rml)=(m^0_\rms+m^0_\rml,m^0_\rmm,0)\in \mathcal{M}_+ \quad\text{and} \quad \ell=-\frac{m^0_\rml}{2}\,.$$
Observe that $\ell$ satisfies
$-\frac{m_\rms}{2}=-\frac{m^0_\rms+m^0_\rml}{2}\leq \ell \leq \frac{m^0_\rms+m^0_\rml}{2}=\frac{m_\rms}{2}+m_\rml$
because $m^0_\rms\geq 0$ and $m^0_\rms+2m^0_\rml\geq 0$.
The inequalities for $\ell$ are strict if $m^0 \in \mathcal{M}_1$ since in this case
$m^0_\rms>0$.

Conversely, suppose that $m^0=m(\ell)$ for $m\in \mathcal{M}_+$ and $\ell$ as in the statement.
Then $-m_\rms\leq 2\ell \leq m_\rms+2m_\rml$ and
$m_\rms+m_\rml=m^0_\rms+m^0_\rml$.
Hence 
$$m^0_\rms=m_\rms+2\ell\leq 2(m_\rms+m_\rml)=2(m^0_\rms+m^0_\rml), \quad \text{i.e.} \quad m^0_\rms+2m^0_\rml\geq 0.$$ Moreover,
$m^0_\rms=m_\rms+2\ell\geq m_\rms-m_\rms=0$.
Thus $m^0 \in \mathcal{M}_+\cup \mathcal{M}_3$.
All the inequalities are strict if $\ell \in \big]\ell_{\rm min}(m),\ell_{\rm max}(m)\big[$. Hence, in this case, $m^0 \in \mathcal{M}_1$.
\end{proof}

Lemma \ref{lemma:M+M3} is pictured in Figure \ref{fig:mell} below. Recall that in this paper we consider the root system $C_r$ as a root system $BC_r$ with $m_\rms=0$.
The diagonal segments belong to lines $m_\rms+m_\rml=\text{constant}$. If $m^0\in \mathcal{M}_+ \cup \mathcal{M}_3$ belongs to such a line, the corresponding element $m\in \mathcal{M}_+$ 
in the first part of the proof, is the intersection of this line with the 
$m_\rms$-axis. The specific segments drawn are those passing through geometric multiplicities $(m_\rms,m_\rmm, m_\rml=1)$: the segments contain the values of $(m_\rms(\ell),m_\rml(\ell)) =(m_\rms+2\ell,1-2\ell)$
with $\ell\in \big[\ell_{\rm min},\ell_{\rm max}]$.

\begin{figure}[h] 
\begin{tikzpicture}[
    scale=.5,
    axis/.style={very thick, ->, >=stealth'},
    equation line/.style={thick},
    equation line dashed/.style={thick, dashed},
    reference line/.style={thin},
    ]
\draw [decorate sep={.8mm}{10mm},fill] (0,1) -- (10,1);
\draw [decorate sep={.5mm}{10mm},fill] (0,0) -- (20,0);
\draw [decorate sep={.4mm}{3mm},fill] (2.5,9.5) -- (5,12);
\draw [decorate sep={.4mm}{3mm},fill] (8.5,3.5) -- (11,6);
\draw [decorate sep={.4mm}{3mm},fill] (16.5,-4.5) -- (19,-2);
\draw [dotted] (2,0) -- (2,1);
\draw [dotted] (4,0) -- (4,1);
\draw [dotted] (6,0) -- (6,1);
\draw [dotted] (8,0) -- (8,1);
\draw [dotted] (10,0) -- (10,1);
\draw [dotted] (12,0) -- (12,1);
\draw [dotted] (14,0) -- (14,1);
\draw [dotted] (16,0) -- (16,1);
\draw [dotted] (18,0) -- (18,0);
\draw [dotted] (20,0) -- (20,0);

\draw[fill=black] (16,1) circle (.8mm);

\draw[equation line dashed] (4,13) -- (6,11);
\draw[equation line dashed] (19,-2) -- (21,-4);

    \draw[axis] (-.5,0)  -- (22,0) node(xline)[right]
        {$m_{\mathrm{s}}$};
    \draw[axis] (0,-.5) -- (0,14.5) node(yline)[above]
        {$m_{\mathrm{l}}$\qquad\null};
     \draw[reference line] (0,0) -- (20,-10)
        node[below left, text width=8em, rotate=-28]
        {$m_{\mathrm{s}}+2m_{\mathrm{l}}=0$};
     \draw[reference line] (0,1) -- (22,1)
        node[above, text width=10em, rotate=0]
        {\null\qquad $m_{\mathrm{l}}=1$};
     \draw[equation line] (0,1) -- (2,-1);
     \draw[equation line] (0,3) -- (6,-3);
     \draw[equation line] (0,5) -- (10,-5);
     \draw[equation line] (0,7) -- (14,-7);
     \draw[equation line] (0,9) -- (18,-9);
     \draw[equation line] (0,11) -- (20,-9);
    \draw[equation line] (6,11) -- (19,-2);

\node[left, rotate=-45] at (0,1) {$\Sigma=C_r$ \;};
\node[left, rotate=-45] at (0,3)
        {$\mathrm{SU}(p,p+1)$ \;};
\node[left, rotate=-45] at (0,5)
        {$\mathrm{SO}^*(2(2n+1))$
        \& $\mathrm{SU}(p,p+2)$ \;};
\node[left, rotate=-45] at (0,7)
        {$\mathrm{SU}(p,p+3)$ \;};
\node[left, rotate=-45] at (0,9)
        {$\mathfrak{e}_{6(-14)}$
        \& $\mathrm{SU}(p,p+4)$ \;};
\node[left, rotate=-45] at (0,11)
        {$\mathrm{SU}(p,p+5)$ \;};
\node[left, rotate=-45] at (4,13)
        {$\mathrm{SU}(p,q)$ \;};

\node at (14,8) {$\mathcal{M}_+$};
\node at (19,-5) {$\mathcal{M}_3$};

\draw (-.4,-.2) node[below,scale=.6] {$0$};
  	\draw (2,-.2) node[below,scale=.6] {$2$};
  	\draw (4,-.2) node[below,scale=.6] {$4$};
  	\draw (6,-.2) node[below,scale=.6] {$6$};
  	\draw (8,-.2) node[below,scale=.6] {$8$};
  	\draw (10,-.2) node[below,scale=.6] {$10$};
  	\draw (16,-.2) node[below,scale=.6] {$2(q-p)$};

\node[right, scale=1] at (7,15) {%
\parbox{8truecm}
{
\begin{tabular}{|c|c|c|c|}
\hline
$G$ & $m_\rms$ & $\ell_{\rm min}$
& $\ell_{\rm max}$\\
\hline
with $\Sigma=C_r$ & 0 & 0 & 1\\
\hline
$\SO^*(2(2n+1))$ & 4 & $-2$ & 3\\
\hline
$\mathfrak{e}_{6(-14)}$ & 8 & $-4$ & 5\\
\hline
$\SU(p,q)$ $q>p$ & $2(q-p)$ & $p-q$ & $q-p+1$\\
\hline
\end{tabular}
}};
\end{tikzpicture}
\caption{$(m_\rms(\ell),m_\rml(\ell))$ for geometric
$(m_\rms,m_\rml=1)$ and $\ell\in[\ell_{\min}, \ell_{\max}]$}
\label{fig:mell}
\end{figure}
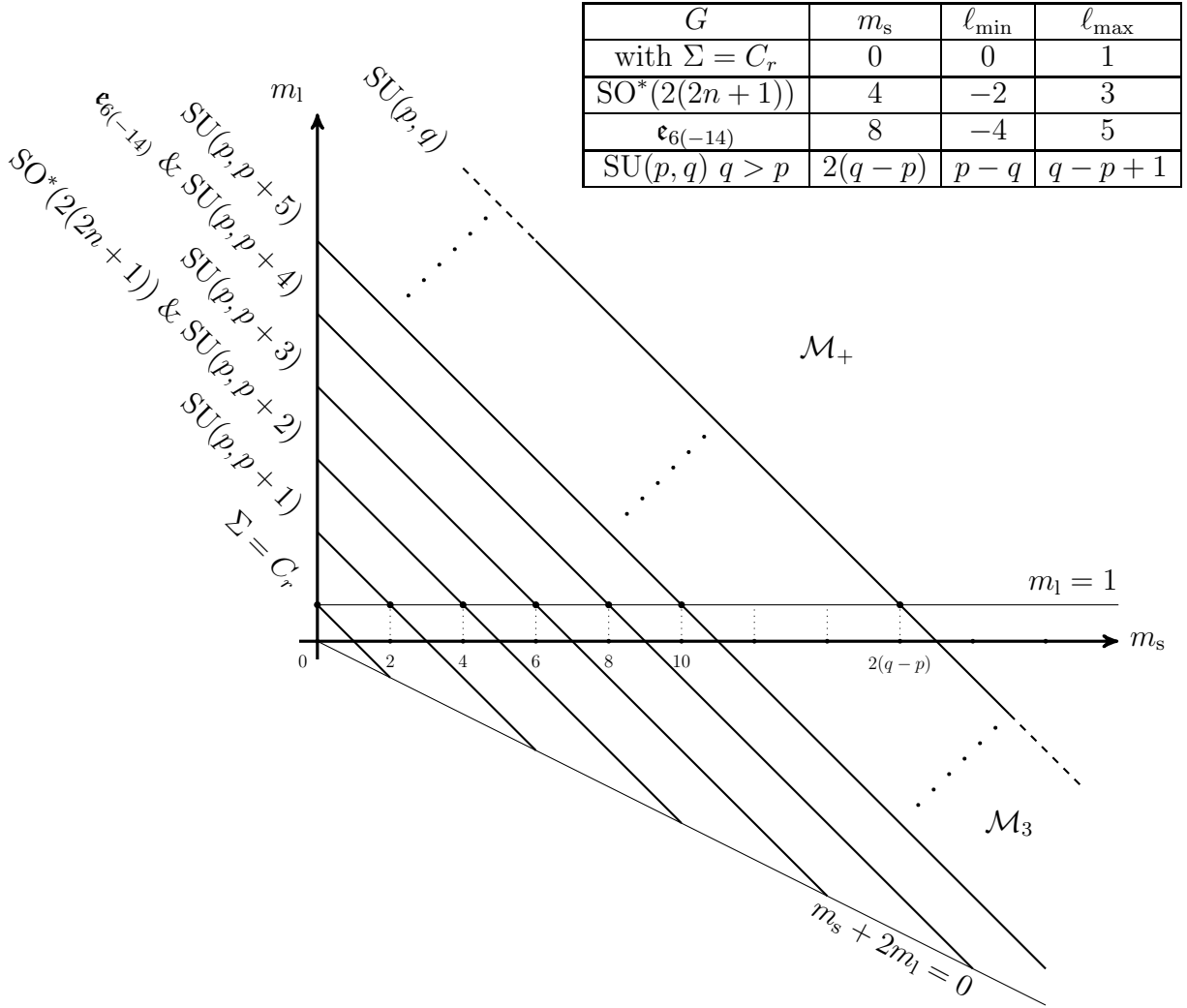


\subsection{$(\ell,\well)$-Cherednik operators}
\label{subsection:Cherednik-two-parameters}
We keep the notation of subsection \ref{subsection:deformation}. Let $u$ and $v$ be the $W$-invariant functions on $\mathfrak{a}$ defined by 
\begin{eqnarray}
\label{eq:u}
&&u(x)=\prod_{j=1}^r  \cosh\Big(\frac{\beta_j(x)}{2}\Big)\,,\\
\label{eq:v}
&&v(x)=\prod_{1\leq i<j\leq r}  \cosh\Big(\frac{\beta_j(x)-\beta_i(x)}{2}\Big)
\cosh\Big(\frac{\beta_j(x)+\beta_i(x)}{2}\Big)\,.
\end{eqnarray}
For $\xi\in \fa$ and $\ell,\well\in \R$ we define the Cherednik operator 
$T_{\ell,\well,\xi}$ by
\begin{equation}
\label{eq:Tellwell}
T_{\ell,\well,\xi}(m)=u^{-\ell} v^{-\well}\circ T_\xi(m(\ell,\well)) \circ u^{\ell}v^{\well}\,.
\end{equation}
A simple computation (using that $u$ and $v$ commute with $1-r_\alpha$ for $\a\in \Sigma$) shows that
$$
T_{\ell,\well,\xi}(m)= T_\xi(m(\ell,\well)) + 
\ell u^{-1} \partial_\xi(u)+\well v^{-1} \partial_\xi(v)\,,
$$
where
\begin{equation}\label{eq:der-u}
u^{-1} \partial_\xi(u)=\frac{1}{2}  \sum_{j =1 }^r
\beta_j(\xi) \tanh \Big( \frac{\beta_j}{2}\Big) 
\end{equation}

and

\begin{equation}\label{eq:der-v}
v^{-1} \partial_\xi(v)=\frac{1}{2}  \sum_{1\leq i<j\leq r}
\Big[(\beta_j(\xi)-\beta_i(\xi)) \tanh \Big( \frac{\beta_j-\beta_i}{2}\Big) +
(\beta_j(\xi)+\beta_i(\xi)) \tanh \Big( \frac{\beta_j+\beta_i}{2}\Big) \Big]\,.
\end{equation}
Let $\mathcal{R}_0$ be the algebra of functions on  $\mathfrak{a}_{\rm reg}$ generated by $1$ and $(1\pm e^{-\alpha})^{-1}$ with $\alpha\in \Sigma^+$; 
see \cite[pp. 63--64]{HC60}. So $\mathcal{R}$ is a subalgebra of $\mathcal{R}_0$. In general, $T_{\ell,\well,\xi} \in  \D_{\mathcal{R}_0} \otimes \C[W]$, where $\D_{\mathcal{R}_0}=
\mathcal{R}_0 \otimes S(\mathfrak{a}_\C)$. However, if $\well=0$, then $T_{\ell,\xi}=T_{\ell,0,\xi} \in  \D_{\mathcal{R}} \otimes \C[W]$.
By construction, 
$\{T_{\ell,\well,\xi}(m) : \xi \in \fa\}$ is a commutative family of differential-reflection operators. So, the map $\xi \to T_{\ell,\well,\xi}(m)$ extends uniquely to an algebra homomorphism $p \to T_{\ell,\well,p}(m)$ from $\polya$ to $\D_{\mathcal{R}_0} \otimes \C[W]$ such that
$T_{\ell,\well,p}(m)=u^{-\ell}v^{-\well} \circ T_p(m(\ell,\well)) \circ v^{\well}u^{\ell}$.
In particular, one can define the $(\ell,\well)$-Heckman-Opdam Laplacian
\begin{equation}
\label{eq:tellHOLaplacian}
T_{\ell,\well,p_L}(m)=\sum_{j=1}^r T_{\ell,\well,\xi_j}(m)^2= u^{-\ell}v^{-\well}\circ T_{p_L}(m(\ell)) \circ u^{\ell}v^{\well}\,,
\end{equation}
where $\{\xi_j\}_{j=1}^r$ is any orthonormal basis of $\fa$ and $p_L$ is defined by $p_L(\l):=\inner{\l}{\l}$ for $\l \in \frakacs$.

Next, we compute the differential part of the $(\ell, \well)$-Heckman-Opdam Laplacian in a closed form, allowing us to deduce some symmetry properties that will be useful in later sections. 
For an arbitrary root system $\Sigma$ on $\mathfrak{a}^*$ and a multiplicity function $m,$ let us set 
\begin{equation}
    \label{fSigma}
 f_\Sigma(m)  = \sum_{\alpha \in \Sigma^+}
\frac{m_\alpha (2-m_\alpha - 2m_{2\alpha}) \langle\alpha, \alpha\rangle}{(e^\alpha - e^{-\alpha})^2}.   
\end{equation}
We start by recalling the following lemma. 

\begin{Lemma}\label{comp-formula1}
Let $\Sigma$ be an arbitrary root system on $\mathfrak{a}^*$ and $m$ a multiplicity function. If $\delta_\Sigma(m)^{\frac{1}{2}} = \prod_{\alpha \in \Sigma^+} (e^\alpha - e^{-\alpha})^{\frac{m_\alpha}{2}},$ then
$$ \delta_{\Sigma}(m)^{\frac{1}{2}} \circ \left ( L_\Sigma(m) + \langle\rho_\Sigma(m), \rho_\Sigma(m)\rangle \right ) \circ \delta_\Sigma(m)^{-\frac{1}{2}}
  = L_{\mathfrak{a}} + f_\Sigma(m), $$
where $L_\Sigma(m)$ is the Heckman-Opdam Laplacian associated to $(\Sigma, m)$ and $\rho_\Sigma(m)$ is the half sum of positive roots.
\end{Lemma}

\begin{proof} See \cite[Theorem 2.1.1]{HS94}
\end{proof}

From now onwards, $\Sigma$ will denote a root system of type $BC$, as in Section \ref{section:notation}.
Recall the operator $\beta$ from \eqref{eq:beta}.
The same definition extends $\beta$ to an operator from 
$\D_{\mathcal{R}_0}\otimes \C[W]$ to $\D_{\mathcal{R}_0}$.
If $p\in \polya$, then $\beta\big(T_{\ell, \well, p}(m)\big)=
u^{-\ell} v^{-\well}\circ \beta\big(T_{p}(m(\ell,\well))\big) \circ u^{\ell} v^{\well}\in \D_{\mathcal{R}_0}$.
Set $D_{\ell, \well,p}(m)=\beta\big(T_{\ell,\well, p}(m)\big)$ for $p\in \polya^W$.
Furthermore, set
\begin{equation}
\label{eq:Dlm}
\D_{\ell, \well}(m)=\{D_{\ell, \well,p}(m): p\in \polya^W\}\,.
\end{equation}

\begin{Lemma}\label{comp-formula2}
Consider the root system $\widetilde{\Sigma} = 2 \Sigma $ and the multiplicity function $\widetilde{m} = (m_\rms+m_\rml, m_\rmm+2\well, 0).$ Then we have the identity 
$$u^{-\frac{m_\rms}{2}} v^{-\frac{m_\rmm}{2}} \circ D_{\ell, \well, p_L}(m) \circ u^{\frac{m_\rms}{2}} v^{\frac{m_\rmm}{2}} = $$
$$L_{\widetilde{\Sigma}} (\widetilde{m}) + \langle\rho_{\widetilde{\Sigma}}(\widetilde{m}), \rho_{\widetilde{\Sigma}}(\widetilde{m})\rangle
 +
f_\Sigma (m(\ell, \widetilde{\ell})) - f_{\widetilde{\Sigma}}(\widetilde{m}).$$
\end{Lemma}

\begin{proof}
We start with
\begin{equation}\label{comp-formula3}
\delta_{\Sigma}(m(\ell, \widetilde{\ell}))^{\frac{1}{2}} \circ \left ( L_\Sigma (m(\ell, \widetilde{\ell})) + \langle\rho (\Sigma, m(\ell, \widetilde{\ell})),
\rho (\Sigma, m(\ell, \widetilde{\ell}))\rangle \right ) \circ \delta_{\Sigma}(m(\ell, \widetilde{\ell}))^{-\frac{1}{2}}
\end{equation}
The above, by Lemma \ref{comp-formula1} equals 
$$ L_{\mathfrak{a}} + \sum_{\alpha > 0 }~\frac{m_\alpha(\ell, \widetilde{\ell}) \left (2-m_\alpha(\ell, \widetilde{\ell})
 - 2m_{2\alpha}(\ell, \widetilde{\ell})\right ) \langle\alpha, \alpha\rangle}{(e^\alpha - e^{-\alpha})^2}  = L_{\mathfrak{a}} + f_{\Sigma}(m(\ell, \widetilde{\ell})).$$ 
 Replacing $L_{\mathfrak{a}}$ with 
 $$\delta_{\widetilde{\Sigma}}(\widetilde{m})^{\frac{1}{2}} \circ \left ( L_{\widetilde{\Sigma}} (\widetilde{m}) + \langle\rho_{\widetilde{\Sigma}}(\widetilde{m}), \rho_{\widetilde{\Sigma}}(\widetilde{m})\rangle \right ) \circ \delta_{\widetilde{\Sigma}} (\widetilde{m})^{-\frac{1}{2}} -f_{\widetilde{\Sigma}}(\widetilde{m}) $$ 
 we get the result as 
 $$\delta_\Sigma(m(\ell, \well))^{\frac{1}{2}} \delta_{\widetilde{\Sigma}}(\widetilde{m})^{-\frac{1}{2}} = 2^M u^{-(\frac{m_s}{2} + \ell)} v^{-(\frac{m_m}{2} + \widetilde{\ell})}\,,$$
 where $M$ is a constant depending on $m$ and $r$.
 \end{proof}
 
\begin{Cor}
\label{cor:symmetry}
In the above notation, for all $l\in \R$ and $\well\geq -\frac{m_\rmm}{2}$,
\begin{equation}
\label{symmetry}
D_{\ell, \well, p_L}(m) = D_{-\ell + m_\rml-1, \well, p_L}\,.
\end{equation}
\end{Cor}
\begin{proof}
A simple computation shows that 
$$
f_{\Sigma}(m(\ell, \well)) =  f_{\Sigma}(m(\well)) + \frac{p^2}{4} \ell (\ell + 1 -m_\rml)\sum_{j=1}^r \frac{1}{\cosh^2 \frac{\beta_j}{2}}\,.
$$ 
The stated equality follows then from Lemma \ref{comp-formula2}. 
\end{proof}

\begin{Prop}
\label{prop:Dellwell}
In the above notation, $\D_{\ell, \well}(m)=u^{-\ell}v^{-\well}\circ \D(m(\ell, \well)) \circ u^{\ell}v^{\well}$ is a commutative subalgebra of 
$\D_{\mathcal{R}_0}^W$ of rank $r$. It is the centralizer of $D_{\ell, \well, p_L}$
in $\D_{\mathcal{R}_0}^W$.
As a consequence, $\D_{\ell, \well}(m)=\D_{-\ell + m_\rml-1, \well}(m)$.
In particular, $\D_{\ell, \well}(m)=\D_{-\ell, \well}(m)$ if $m_\rml=1$.

If $\well=0$, then $\mathcal{R}_0$ can be replaced by $\mathcal{R}$.
Furthermore, if $m=(m_\rms,m_\rmm,m_\rml=1)$ is geometric and $\ell\in \Z,$ $\well = 0,$ then $\D_{\ell, 0}(m)=\delta_{\ell}(\D(G/K; \tau_{\ell}))$.
\end{Prop}
\begin{proof}
Because of \eqref{eq:der-u}, \eqref{eq:der-v} and since $u$ and $v$ are $W$-invariant, one can easily check that the conjugation by $u^{\ell} v^{\well}$ maps 
$\D_{\mathcal{R}_0}^W$ into itself 
and $\D_{\mathcal{R}}^W$ into itself if $\well=0$.
Now the proof follows from the more general result in \cite[Lemma 5.2]{OS17}, 
showing that the centralizers of 
 $D_{p_L}(m(\ell,\well))$ in $\D_{\mathcal{R}_0}^W$ and $\D_{\mathcal{R}}^W$ agree, together with
\eqref{symmetry}. Here, $\delta_{\ell}(\D(G/K, \tau_{\ell})$ denotes the $\tau_{\ell}$-radial components of the commutative algebra $\D(G/K; \tau_{\ell})$ (see subsection \ref{subsection:tauell-spherical}).
\end{proof}

\subsection{$\tauwell$-Hypergeometric functions}
Let $\l\in \fa^*_\C$ be fixed. The \textit{$\tauwell$-Heckman-Opdam
hypergeometric function with spectral parameter $\l$} is the unique
$W$-invariant analytic solution $F_{\ell, \well,\l}(m)$ of the system of
differential equations 
\begin{equation}
  \label{eq:hypereq3}
D_{\ell, \well,p}(m) f=p(\l)f \qquad (p \in \polya^W),
\end{equation}
which satisfies $f(0)=1$.

The \textit{non-symmetric $\tauwell$-hypergeometric function with
spectral parameter $\l$} is the unique analytic solution
$G_{\ell,\well, \l}(m)$ of the system of differential equations 
\begin{equation}
  \label{eq:hypereq4}
T_{\ell,\well,\xi}(m) g=\l(\xi) g \qquad (\xi \in \fa),
\end{equation}
which satisfies $g(0)=1$.

The symmetric and non-symmetric $\tauwell$-Heckman-Opdam
hypergeometric functions are therefore (suitably normalized) joint
eigenfunctions of the commutative algebras $\D_{\ell, \well}(m)$ and
$\{T_{\ell, \well,p}:p\in S(\fa_\C)\}$, respectively. Notice that the
equality for $\D_{\ell, \well}(m)$ in Proposition \ref{prop:Dellwell} yields
\begin{equation}
\label{eq:Fellm-even-ell}
F_{\ell,\well, \l}(m)=F_{-\ell+m_\rml-1,\well,\l}(m)\qquad (\l\in\fa_\C^*)\,.
\end{equation}
On the other hand, an analogous symmetry relation is generally not true for $G_{\ell, \well,\l}(m)$ as can be seen from rank one examples. We omit the details.

By definition, $F_{\ell,\well,\l}(m;x)$ is $W$-invariant in $x\in \fa$ and in $\l\in \fa^*_\C$.
Furthermore, since $u^\ell v^{\well} \circ D_{\ell, \well,p}(m)\circ u^{-\ell}v^{-\well}=D_{p}(m(\ell, \well))$ and
$u^\ell v^{\well}\circ T_{\ell,\well,\xi}(m)\circ u^{-\ell}v^{-\well}=T_{\xi}(m(\ell, \well))$, one obtains for all $\l\in \fa_\C^*$:
\begin{eqnarray}
\label{eq:Fell-F}
F_{\ell,\well,\l}(m)&=&u^{-\ell} v^{-\well}F_{\l}(m(\ell, \well))\,,\\
G_{\ell,\well,\l}(m)&=&u^{-\ell}v^{-\well} G_{\l}(m(\ell, \well))\,.
\label{eq:Gell-G}
\end{eqnarray}
As in the case $\ell=\well=0$,
\begin{equation}
\label{eq:Fell-Gell}
F_{\ell,\well,\l}(m;x)=\frac{1}{|W|}\, \sum_{w\in W} G_{\ell,\well,\l}(m;w^{-1}x) \qquad (x \in \fa)\,.
\end{equation}
As a consequnce of \eqref{eq:Fell-F} and of the corresponding properties of the Heckman-Opdam hypergeometric functions (see e.g. \cite[Theorem 4.4.2]{HS94}, \cite[Theorem 3.15]{OpdamActa}), the
$F_{\ell,\well, \l}(m;x)$ and $G_{\ell,\well,\l}(m;x)$ are entire functions of $\lambda \in \fa_\C^*$, analytic functions in $x\in \fa$
and meromorphic functions of $m=(m_\rms,m_\rmm,m_\rml) \in \C^3$.  Let $\mathcal M_{\C,0}=\{m=(m_\rms,m_\rmm,m_\rml)\in \C^3:\Re(m_\rms+m_\rml)\geq 0 \}$.
Observe that $m\in \mathcal M_{\C,0}$ if and only if $m(\ell)\in  \mathcal M_{\C,0}$.  One can show (see Appendix A below) that for fixed $(\l,x)\in \fa^*_\C \times \fa$, the functions $F_\l(m;x)$ and $G_\l(m;x)$ are holomorphic in a neighborhood of $\mathcal M_{\C,0}$. It follows that $F_{\ell,\well,\l}(m;x)$ and $G_{\ell,\well, \l}(m;x)$ are holomorphic near each $m=(m_\rms,m_\rmm,m_\rml)\in \mathcal M_0$.

\subsection{Estimates and asymptotics for the $\tauwell$-hypergeometric functions}
\label{subsection:taul}
Let $m=(m_\rms,m_\rmm,m_\rml)\in \mathcal{M}_+$.
Using \eqref{eq:Fell-F} and \eqref{eq:Gell-G}, we can obtain from the results of section \ref{section:basic} some estimates and asymptotic properties for the $\tauwell$-hypergeometric functions. The factors $u^{-\ell}$ and $v^{-\well}$ require some additional care for the asymptotics.

By Lemma \ref{lemma:M+M3}, we have $m(\ell, \well)\in
\mathcal{M}_+\cup \mathcal{M}_3$ for $m\in \mathcal{M}_+$ and
$\ell\in \big[\ell_{\rm min}(m),\ell_{\rm max}(m)\big], \well \geq -\frac{m_\rmm}{2}$.
This implies that Proposition \ref {prop:positivity-estimates-M3}, Theorem \ref{thm:hc-expansion} and
Corollary \ref{cor:leading-termF-realcase} hold true for
$F_\l(m(\ell, \well))$ for all $m\in \mathcal{M}_+$ and $\ell\in \big]\ell_{\rm min}(m),\ell_{\rm max}(m)\big[$. In turn, \eqref{eq:Fell-F} yields analogous statements for the $\tauwell$-hypergeometric functions.
Recall that $\ell_{\rm min}=-\frac{m_\rms}{2}$ and $\ell_{\rm max}=\frac{m_\rms}{2}+m_\rml$. (We are omitting from the notation the dependence on $m$ of $\ell_{\rm min}$ and $\ell_{\rm max}$). So
$\ell_{\rm min} \leq 0 < \ell_{\rm max}$.

Recalling also that $F_{\ell,\well,\lambda}=F_{-\ell+\rml-1,\well,\lambda}$, we see that we can extend the inequalities for $F_{\ell,\well,\lambda}$ to $\ell \in 
]-\frac{m_\rms}{2}-1, \frac{m_\rms}{2}+m_\rml[$ (or to $\ell \in [-\frac{m_\rms}{2}-1, \frac{m_\rms}{2}+m_\rml]$ where the extension by continuity in the multiplicity parameter is possible; see Theorem \ref{thm:reg-prop} in the appendix). For a fixed multiplicity $m=(m_\rms,m_\rmm,m_\rml)$, the symmetries of $F_{\ell,\well,\lambda}$ in the $\ell$-parameter can be pictured as
symmetries of the deformations $m(\ell)$ around the point $m(\frac{m_\rml-1}{2})$, as in Figure \ref{fig:symmetries-ell}.

Similarly, the inequalities of
Lemma \ref{lemma:OpdamM2M3} extend to any $\ell\in \R$ because either the condition $\ell\leq \ell_{\max}$ or $-\ell+m_\rml-1 \leq \ell_{\max}$ is always satisfied. 
In the first case, $m(\ell,\well)\in \mathcal{M}_2\cup \mathcal{M}_3$; in the other, $m(-\ell+m_\rml+1,\well)\in \mathcal{M}_2\cup \mathcal{M}_3$.

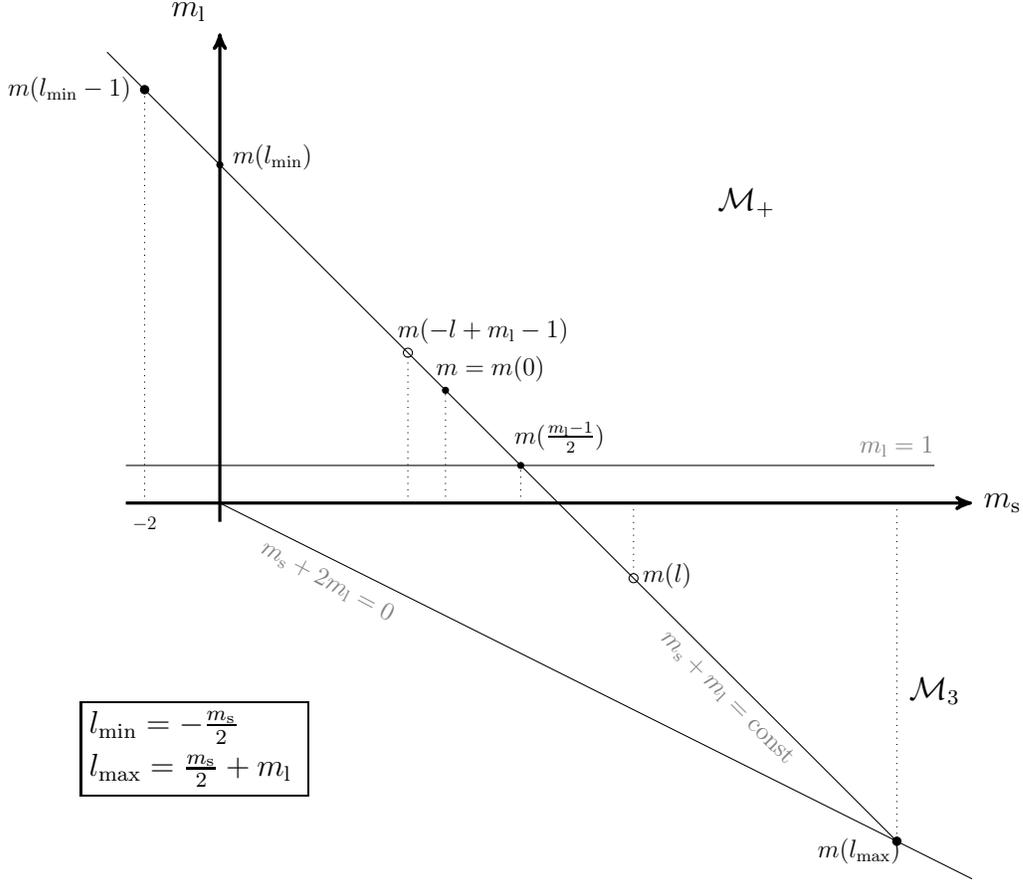
\begin{figure}[h]
\begin{tikzpicture}[
    scale=.5,
    axis/.style={very thick, ->, >=stealth'},
    equation line/.style={thick},
    equation line dashed/.style={thick, dashed},
    reference line/.style={thin},
    ]
\draw [dotted] (-2,0) -- (-2,11);
\draw [dotted] (5,0) -- (5,4);
\draw [dotted] (6,0) -- (6,3);
\draw [dotted] (8,0) -- (8,1);
\draw [dotted] (11,0) -- (11,-2);
\draw [dotted] (18,0) -- (18,-9);

\draw[fill=black] (18,-9) circle (1.1mm);
\draw (11,-2) circle (1.2mm);
\draw[fill=black] (8,1) circle (.8mm);
\draw[fill=black] (6,3) circle (.8mm);
\draw (5,4) circle (1.2mm);
\draw[fill=black] (0,9) circle (.8mm);
\draw[fill=black] (-2,11) circle (1.1mm);

    \draw[axis] (-2.5,0)  -- (20,0) node(xline)[right]
        {$m_{\mathrm{s}}$};
    \draw[axis] (0,-.5) -- (0,12.5) node(yline)[above]
        {$m_{\mathrm{l}}$\qquad\null};
     \draw[reference line] (20,-10) -- (0,0) 
        node[below right, text width=8em, rotate=-28, scale=.8, color=gray]
        {\null\qquad $m_{\mathrm{s}}+2m_{\mathrm{l}}=0$};
     \draw[reference line] (-2.5,1) -- (19,1)
        node[above, text width=10em, rotate=0, scale=.8, color=gray]
        {\null\qquad $m_{\mathrm{l}}=1$};
     \draw[reference line] (-3,12) -- (18,-9) 
   node[below, text width=25em, rotate=-45,scale=.8, color=gray]
        {$m_{\mathrm{s}}+m_{\mathrm{l}}={\rm const}$ };    
        

\node at (14,8) {$\mathcal{M}_+$};
\node at (19,-5) {$\mathcal{M}_3$};
\draw (17,-8.7) node[below, scale=.8] {$m(l_{\max})$};
\draw (11.9,-2.5) node[above, scale=.8] {$m(l)$};
\draw (7.6,2.4) node[below right, scale=.8] {$m(\frac{m_{\rml}-1}{2})$};
\draw (7.2,3) node[above, scale=.8] {$m=m(0)$};
\draw (7,4) node[above,scale=.8] {$m(-l+m_\rml-1)$};
\draw (1.4,9.2) node[scale=.8] {$m(l_{\min})$};
\draw (-4,11) node[scale=.8] {$m(l_{\min}-1)$};
\draw (-2,-.2) node[below,scale=.6] {$-2$};

\draw (-4,-5) node[below right]
{\fbox{\parbox{2.8cm}{$l_{\min}=-\frac{m_\rms}{2}$\\[.2em]
$l_{\max}=\frac{m_\rms}{2}+m_\rml$}}};
\end{tikzpicture}
\vskip -2cm

\caption{Symmetries of $m(\ell)$ 
around $m(\frac{m_\rml-1}{2})$ if $m_\rml\geq 1$}
\label{fig:symmetries-ell}
\end{figure}

We leave to the reader the simple task of modifying the statements of Theorem \ref{thm:hc-expansion} and
Corollary \ref{cor:leading-termF-realcase} in this case, and
we collect the estimates obtained for $G_{\ell,\well, \lambda}(m)$ and $F_{\ell, \well, \lambda}(m)$ in the corollary below.

\begin{Cor}
\label{cor:pos-est-ell}
Let $m=(m_\rms,m_\rmm,m_\rml) \in \mathcal{M}_+$ be a nonnegative multiplicity function
on a root system of type $BC_r$, and let $\ell, \well \in\R,$ $\well \geq -\frac{m_\rmm}{2}$. Then the following properties hold.
\begin{enumerate}
\thmlist
\item For all $\ell \leq\ell_{\rm max}(m)$,  $\l\in\fa_\C^*$ and $x\in \mathfrak{a}$,
\begin{equation}
\label{eq:b-e1-G}
|G_{\ell, \well, \l}(m;x)|\leq \sqrt{|W|} \, u^{-\ell}(x) v^{-\well}(x) e^{\max_w (w\l)(x)}
\end{equation}
and for all $\ell\in\R$, $\l\in\fa_\C^*$ and $x\in \mathfrak{a}$,
\begin{equation}
\label{eq:b-e1-F1}   
|F_{\ell, \well, \l}(m;x)|=|F_{-\ell+m_\rml-1,\l}(m;x)|\leq \sqrt{|W|} \, u^{-\ell}(x) v^{-\well}(x) e^{\max_w (w\l)(x)}\,.
\end{equation}

\item 
Suppose $\ell \in [\ell_{\rm min}-1,\ell_{\rm max}]$.
For all
$\l\in\fa^*$ the function $F_{\ell,\well, \l}(m)$ is real and strictly positive.
For all $\l\in\fa_\C^*$
\begin{equation}
\label{eq:b-e1-F2}
|F_{\ell, \well, \l}(m)|\leq F_{\ell,\well, \Re\l}(m) \,.\\
\end{equation}
Moreover, for all $\l\in\fa^*$ and all $x\in \fa$
\begin{equation}
\label{eq:b-e2-F}
F_{\ell,\well, \l}(m;x)\leq F_{\ell,\well, 0}(m;x)e^{\max_w (w\l)(x)}\,.
\end{equation}
More generally, for all $\l\in\fa^*$, $\mu\in \overline{(\fa^*)^+}$ and all $x\in \fa$
\begin{equation}
\label{eq:b-e3-F}
F_{\ell,\well, \l+\mu}(m;x)\leq F_{\ell,\well, \mu}(m;x)e^{\max_w (w\l)(x)} \,.
\end{equation}
The same properties hold for $G_{\ell, \well, \l}(m)$ provided $\ell \in [\ell_{\min}, \ell_{\max}].$

\item 
Suppose $\ell \in 
[\ell_{\min}-1,\ell_{\max}]$, 
$\well\geq 0$, $\l \in \overline{(\mathfrak{a}^*)^+}$ and $x_1 \in
\mathfrak{a}^+$. Further assume that $m_\rml \geq 1.$ Then for all  $x \in \mathfrak{a}$ we have
\begin{equation} \label{eq:subadd-Schapira-special-well}
F_{\ell,\well, \l}(m; x+x_1) e^{-(\l+\rho(m(2\well)))(x_1)} \leq
F_{\ell,\well, \l}(m; x) \leq F_{\ell,\well, \l}(m; x+x_1)
e^{(\l+\rho(m(2\well)))(x_1)}\,.
\end{equation}
\item 
Suppose $\ell \in ]\ell_{\min}, \ell_{\max}[$, $\well > 0$ if $m_\rmm =0,$ $\well \geq 0$ if $m_\rmm > 0$ and $\l \in \overline{(\mathfrak{a}^*)^+}$. Then for all $x \in
\overline{\mathfrak{a}^+}$ we have
\begin{eqnarray}
\label{asymp-ell-well}
F_{\ell,\well, \l}(m; x) &\asymp &\big[\prod_{\a\in\Sigma_{\l}^0}
(1+\a(x))\big] e^{(\l-\rho(m(2\well)))(x)}\,,
\end{eqnarray}
where $\Sigma_{\l}^0=\{ \a\in \Sigma_\rms^+\cup \Sigma_\rmm^+:\inner{\a}{\l}=0\}$.
The asymptotics \eqref{asymp-ell-well} extend to $\ell \in ]\ell_{\min}-1, \ell_{\max}[$
if $m_\rms+m_\rml\geq 1$.
\end{enumerate}
\end{Cor}

\begin{proof}
We prove (c) and (d). For (c), first assume that $\ell \in [\ell_{\min}, \ell_{\max}].$ since $F_{\ell, \well, \lambda} = F_{-\ell+m_\rml-1, \well, \lambda}$ we may assume that $\ell \geq \frac{m_\rml-1}{2} \geq 0.$
To keep track of the exponents appearing in our formulas, it will be convenient to introduce the following notation, 
\begin{equation}\label{eq:sum-middle}
\sum_{1 \leq i < j \leq r} (\beta_j \pm \beta_i) = \sum_{1 \leq i < j \leq r} (\beta_j + \beta_i) + \sum_{1 \leq i < j \leq r} (\beta_j - \beta_i) = \sum_{j =1}^r 2(j-1) \beta_j.
\end{equation}
For $a\geq 0$ and $b\geq 0$,
$$
\cosh a \cosh b\leq \cosh(a+b)\leq \cosh a\;  e^b\,.
$$
Hence, if $\ell\geq 0$, $\well\geq 0$ and $x,x_1$ are as above,
\begin{align}
\label{duplication-uv}
\begin{split}
u^{\ell}(x)u^{\ell}(x_1)  &\leq  u^{\ell}(x+x_1)
\leq  u^{\ell}(x)e^{\frac{\ell}{2}\sum_{j=1}^r \beta_j(x_1)} \,,\\
v^{\well}(x)v^{\well}(x_1) &\leq  v^{\well}(x+x_1)
\leq   v^{\well}(x)e^{\frac{\well}{2} \sum_{1 \leq i < j \leq r} (\beta_j \pm \beta_i)(x_1)}  \,.
\end{split}
\end{align}

Notice that
\begin{equation}
\label{rho2well}
\rho(m(2\well))=\rho(m(\ell,\well))+\frac{\ell}{2}\sum_{j=1}^r \beta_j + \frac{\well}{2} \sum_{1 \leq i < j \leq r}(\beta_j \pm \beta_i)\,.
\end{equation}
By \eqref{eq:subadd-Schapira-special}, which applies since $m(\ell,\well)\in \mathcal{M}_3$,
$$
F_\l(m(\ell,\well); x+x_1) e^{-(\l+\rho(m(\ell,\well)))(x_1)}\leq
F_\l(m(\ell,\well); x) \leq
F_\l(m(\ell,\well); x+x_1) e^{(\l+\rho(m(\ell,\well)))(x_1)}\,.
$$
Multiplying each term of the inequality by $u(x)^{-\ell}v(x)^{-\well}$ and using the definition of $F_{\ell,\well, \lambda}$, we obtain
\begin{multline*}
\frac{u^{\ell}(x+x_1) v^{\well}(x+x_1)}{u^{\ell}(x) v^{\well}(x)}
F_{\ell,\well, \lambda}(m; x+x_1) e^{-(\l+\rho(m(\ell,\well)))(x_1)}
\\
\leq F_{\ell,\well, \lambda}(m; x) \leq
\frac{u^{\ell}(x+x_1) v^{\well}(x+x_1)}{u^{\ell}(x) v^{\well}(x)}
F_{\ell,\well, \lambda}(m; x+x_1) e^{(\l+\rho(m(\ell,\well)))(x_1)}\,.
\end{multline*}
The right hand side inequality in (c) follows from the second inequalities in \eqref{duplication-uv} and
\eqref{rho2well}. Again, using \eqref{rho2well} (and \eqref{eq:sum-middle}), the expression on the left hand side inequality of (d) becomes 
$$ \frac{u^\ell (x+x_1) v^{\well}(x+x_1)}{u^\ell(x_1) v^{\well}(x_1)} e^{\frac{\ell}{2} \sum_{j=1}^r \beta_j(x_1)} e^{\frac{\well}{2} \sum_{1 \leq i < j \leq r} (\beta_j \pm \beta_i)(x_1)} F_{\ell, \well, \lambda}(m; x+x_1) e^{-(\lambda + \rho(m(2\well)))(x_1)}  .$$ This implies (c) as
$$\frac{u^\ell (x+x_1) v^{\well}(x+x_1)}{u^\ell(x_1) v^{\well}(x_1)} e^{\frac{\ell}{2} \sum_{j=1}^r \beta_j(x_1)} e^{\frac{\well}{2} \sum_{1 \leq i < j \leq r} (\beta_j \pm \beta_i)(x_1)} \geq 1.$$ For $\ell \in [\ell_{\min}-1, \ell_{\min}[,$ notice that $\ell_{\min} \leq -\ell+m_\rml-1 \leq \ell_{\max}$ as $m_\rml \geq 1$ and use the equality $F_{\ell, \well, \lambda} = F_{-\ell+m_\rml-1, \well, \lambda}.$
Next we prove (d). For $a\geq 0$ we have $\cosh a \asymp e^a$. Hence
\begin{align}
\label{trivial estimate}
\begin{split}
u^{-\ell}(x) &\asymp e^{-\frac{\ell}{2}\sum_{j=1}^r \beta_j(x)} \\
v^{-\well}(x) &\asymp e^{\frac{-\well}{2} \sum_{1 \leq i < j \leq r} (\beta_j \pm \beta_i)(x)}\,.
\end{split}
\end{align}
Since $m(\ell,\well)\in \mathcal{M}_3^0$, we have by Theorem \ref{thm:real},
$$
F_{\lambda}(m(\ell,\well);x)\asymp \big[\prod_{\a\in\Sigma_{\l}^0}
(1+\a(x))\big] e^{(\l-\rho(m(\ell,\well)))(x)}\,.
$$
Then \eqref{asymp-ell-well} follows by multiplying both sides of this asymptotics by
$u^{-\ell}(x)v^{-\well}(x)$, together with \eqref{trivial estimate} and \eqref{rho2well}.

If $\ell\in ]\ell_{\min}-1,\ell_{\min}[$, then $-\ell+m_\rml-1\in ]-\ell_{\min}+m_\rml-1,
\ell_{\max}[$. The lower bound satisfies $-\ell_{\min}+m_\rml -1 \geq \ell_{\min}$
if and only if $m_\rms+m_\rml\geq 1$. In this case, the above asymptotics hold for
$-l+m_\rml-1$ as well. They lead to \eqref{asymp-ell-well}
using $F_{\ell,\well,\lambda}(m)=F_{-\ell+m_\rml-1,\well,\lambda}(m)$.
\end{proof}

\subsection{Bounded $\tauwell$-hypergeometric functions}

In this section we address the problem of characterizing the
$\tauwell$-hypergeometric functions which are bounded.
Recall that we are considering a multiplicity function $m = (m_\rms, m_\rmm, m_\rml)\in \mathcal{M}_+$ and its
deformation $m(\ell, \well)$ as in \eqref{eq:mult-l-tl-gen}. To apply the estimates from subsection \ref{subsection:taul}, we will assume that $\ell \in [\ell_{\min}, \ell_{\max}]$. The asymptotics from Theorem \ref{thm:hc-expansion} hold under the stronger assumption that $\ell \in ]\ell_{\min}, \ell_{\max}[$.

Recall also that
$$\rho(m(\ell, \well))=\frac{1}{2} \sum_{\a \in \Sigma^+} m_\a(\ell, \well) \a=
 \rho(m)-\frac{\ell}{2}\, \sum_{j=1}^r \beta_j + \frac{\well}{2} \sum_{1 \leq i < j \leq r} (\beta_j \pm \beta_i)$$
 where $\sum_{1 \leq i < j \leq r} (\beta_j \pm \beta_i)$ is given by \eqref{eq:sum-middle} and so, $$ \rho(m(2\well)) = \rho(m(\ell, \well))+\frac{\ell}{2} \sum_{j=1}^r \beta_j + \frac{\well}{2} \sum_{1 \leq i < j \leq r} (\beta_j \pm \beta_i) .$$ 
Notice that, if we suppose $\ell \in [0,\ell_{\rm max}],\well \geq -\frac{m_\rms}{2}$ then $m(\ell, \well)\in \mathcal{M}_+\cup \mathcal{M}_3$ by Lemma \ref{lemma:M+M3}. Hence  $\rho(m(\ell, \well))\in \overline{(\mathfrak{a}^*)^+}$. Notice also that $\rho(m(2\well)) \geq \rho(m)$ if $\well \geq 0.$

\begin{Thm}
\label{thm:bdd}
Assume that $m_\rml \geq 1$ and $\well \geq 0$ if $m_\rmm > 0,$ $\well > 0$ if $m_\rmm =0.$
Fix $\ell \in \R$ with $\ell \in ]\ell_{\rm min}-1, \ell_{\rm max}[$. Then, the $\tauwell$-hypergeometric
function $F_{\ell, \well, \l}(m)$ is bounded if and only if $\l \in
C(\rho(m(2\well))) + i \fa^\ast,$ where $C(\rho(m(2\well)))$ is the convex hull of
the set $\{w\rho(m(2\well)):~w \in W\}.$ 
\end{Thm}
\begin{proof}
First note that $m(\ell, \well) \in \mathcal{M}_1.$ Since $F_{\ell, \well, \l} = F_{-\ell+m_\rml-1, \well, \l}$ we may assume that $ 0 \leq \frac{m_\rml -1}{2} \leq \ell < \ell_{\rm max} $. We show that $F_{\ell, \well, \l}$ is bounded if $\l \in
C(\rho(m(2\well)))+i \fa^\ast.$  Let $\l \in C(\rho(m(2\well)))+i \fa^\ast$ and consider the holomorphic function $\l \to F_{\ell, \well, \l}(m ; x)$
(for a fixed $x$). Since $$ |F_{\ell, \well, \l}(m ; x)| \leq F_{\ell, \well,
\Re \l}(m ; x) = u^{-\ell}(x) v^{-\well}(x) F_{\Re \l}(m(\ell, \well); x),$$ we can
argue as in the proof of \cite[Theorem 4.2]{NPP} to obtain that
the maximum of $|F_{\ell, \well, \l}(m; x)|$ is attained at
$\{w\rho(m(2\well)):~w \in W\}.$ That is
$$
|F_{\ell, \l}(m; x)| \leq u^{-\ell}(x) v^{-\well}(x) F_{\rho(m(2\well))}(m(\ell, \well); x).
$$
As noticed before,  $\rho(m(\ell, \well))\in \overline{(\fa^\ast)^+}.$
Applying \eqref{eq:basic-estimate3}, we obtain
\begin{multline*}
F_{\rho(m(2\well))}(m(\ell, \well); x) =
F_{\rho(m(\ell, \well)) + \frac{\ell}{2} \sum_{j=1}^r \beta_j + \frac{\well}{2} \sum_{1 \leq i < j \leq r} (\beta_j \pm \beta_i)}(m(\ell, \well);
x) \\
\leq F_{\rho(m(\ell, \well))}(m(\ell, \well) ; x)e^{\max_w w \left [ \frac{\ell}{2}
\sum_{j=1}^r \beta_j(x) + \frac{\well}{2} \sum_{1 \leq i < j \leq r} (\beta_j \pm \beta_i)(x) \right ]}.
\end{multline*}
Now, $F_{\rho(m(\ell, \well))}(m(\ell, \well); x) =1$; see e.g. \cite[Lemma 4.1]{NPP} (the proof extends to every multiplicity function for which the symmetric and non-symmetric hypergeometric functions are defined).
Choose $x_0$ in the $W$-orbit of $x$ so that
$$
\max_w w \left [ \frac{\ell}{2}
\sum_{j=1}^r \beta_j(x) + \frac{\well}{2} \sum_{1 \leq i < j \leq r} (\beta_j \pm \beta_i)(x) \right ] = \frac{\ell}{2} \sum_{j=1}^r
\beta_j(x_0) + \frac{\well}{2} \sum_{1 \leq i < j \leq r} (\beta_j \pm \beta_i)(x_0)
$$
Since $u$ and $v$ are $W$-invariant, from the above we get
$$
|F_{\ell, \well, \l}(m; x)| \leq u^{-\ell}(x_0) v^{-\well}(x_0)
e^{ \frac{\ell}{2}\sum_{j=1}^r \beta_j(x_0)} e^{\frac{\well}{2} \sum_{1 \leq i < j \leq r}(\beta_j \pm \beta_i)(x_0)} \leq 2^M
$$
where $M$ is a constant depending on $\ell, \well$ and $r.$

To prove the
other way, we proceed as in \cite{NPP} (see pages 251--252). Let
$\lambda_0$ be such that $\Re \lambda_0 \in
\overline{(\fa^\ast)^+} \setminus C(\rho(m(2\well))).$ Let $x_1 \in \fa^+$
be such that $(\Re \lambda_0 - \rho(m(2\well)))(x_1) > 0.$ If $F_{\ell, \well,
\lambda_0}(m)$ is bounded we have $$ \lim_{t \to \infty} F_{\ell, \well,
\lambda_0}(m;tx_1) e^{-t(\Re \lambda_0 -\rho(m(2\well)))(x_1)} t^{-d} = 0.$$
Since, $F_{\ell, \well \lambda_0}(m;x)  = u^{-\ell}(x) v^{-\well}(x) F_{\lambda_0}(m(\ell, \well);
x)$ and $\cosh a$ behaves like $e^a$ for large positive $a,$
we have
$$ F_{\lambda_0}(m(\ell, \well); tx_1) e^{-t(\Re \lambda_0 - \rho(m(\ell, \well)))(x_1)}
t^{-d} \to 0 \qquad t \to \infty.$$
So, by Theorem \ref{thm:hc-expansion} (which is applicable as $m(\ell, \well) \in \mathcal{M}_1$), the proof can be completed as in Theorem 4.2 of \cite{NPP}.
\end{proof}

\begin{Rem}
The proof of Theorem \ref{thm:bdd} shows that the $\tauwell$-hypergeometric
function $F_{\ell, \well, \l}(m)$ is bounded for $\l \in C(\rho(m(2\well))) + i \fa^\ast$, provided $\ell \in [\ell_{\rm min}-1, \ell_{\rm max}]$ and $\well \geq 0.$
For general root systems of type $BC_r$, we cannot prove that when
$\ell=\ell_{\rm max}$ the function $F_{\ell, \well, \l}(m)$ is not bounded for $\l \notin
C(\rho(m(2\well))) + i \fa^\ast$. This is because the function $b_0$ appearing in Theorem \ref{thm:hc-expansion} might vanish. However the above theorem continues to hold for $\ell = \ell_{\rm max}$ provided $r =1.$
Indeed, for $\ell=\ell_{\rm max}$ we have $m_\rms(\ell)+2m_\rml(\ell)=0$ and, by \eqref{eq:b-nonzero},  the function $b_0(m(\ell_{\rm max});\lambda_0)$ vanishes for $\lambda_0 \in \frakacs$ with $\Re \lambda_0 \in \overline{(\fa^\ast)^+},$ if and only if there is $j\in \{1,\dots,r\}$ such that
$(\lambda_0)_{\beta_j}=0$. Outside these $\lambda_0$, the asymptotics from Theorem \ref{thm:hc-expansion} still hold. In the rank-one case, there is
only one $\lambda_0$ for which the asymptotics do not hold, namely $\lambda_0=0$ which is not outside $C(\rho(m)) + i \fa^\ast$. So the above result holds under the assumption $\ell \in [\ell_{\rm min}, \ell_{\rm max}],$ in the rank one case.
\end{Rem}

\begin{Rem} 
It is natural to expect that $F_{\ell, \well, \lambda}$ is bounded by one if $\lambda \in C(\rho(m(2\well))) + i \fa^\ast$, though we are not able to prove this. However, this holds true for the geometric case as can be seen from Corollary \ref{Cor:estimates-spherical} and from the integral formula \eqref{integral-formula-tau} together with the case-by-case computation that $\rho_{G/K}=\rho(m(2\well))$; see subsection \ref{subsection:smallKtype}.
\end{Rem}

\section{Some geometric cases}
\label{section:geometric}

\subsection{Spherical functions on line bundles over Hermitian symmetric spaces}
\label{subsection:tauell-spherical}

In this subsection we shall assume that $\well =0$ and consequently suppress the index $\well$ from the notation.

Let $m=(m_\rms,m_\rmm,m_\rml=1)$ be a geometric multiplicity function
corresponding to a non-compact Hermitian symmetric space $G/K$, as in Example \ref{rem:hermitian}.
Let $\tau_\ell$ be a fixed one dimensional unitary representations of $K$. So 
$\ell\in \Z$ (or $\ell \in \R$ by passing to universal covering spaces).
Let $E_\ell$ denote the homogeneous line bundle on $G/K$ associated with $\tau_\ell$. The smooth $\tau_{\ell}$-spherical sections of $E_\ell$ can be identified with the space $C^\infty(G//K;\tau_\ell)$ of functions on $G$ satisfying
$f(k_1gk_2)=\tau_{-\ell}(k_1k_2)f(g)$ for all $g\in G$ and $k_1,k_2\in K$. 
Recall our convention of identifying the Cartan subspace $\mathfrak{a}$ with its diffeomorphic image $A=\exp(\mathfrak{a})\subset G$. So every element $\varphi\in C^\infty(G//K;\tau_\ell)$ is uniquely determined by its Weyl-group invariant restriction $\varphi|_{\mathfrak{a}}$ to $A=\mathfrak{a}$. 
Recall also the notation 
$\D(G/K;\tau_\ell)$ for the (commutative) algebra of the $G$-invariant differential operators on $E_\ell$. Then the $\tau_{\ell}$-spherical functions on $G/K$ are the (suitably normalized) joint eigenfunctions of $\D(G/K;\tau_\ell)$ belonging to $C^\infty(G//K;\tau_\ell)$.  For a fixed $\ell$, they are parametrized by $\lambda\in \mathfrak a_\C^*$ (modulo the Weyl group). We denote by $\varphi_{\ell,\lambda}$ the $\tau_{\ell}$-spherical function of spectral parameter $\lambda$.

Many authors have studied the $\tau_{\ell}$-spherical functions by relating them to hypergeometric functions on root systems. See e.g. \cite[\S 5.2--5.5]{HS94}, \cite[\S 6]{ShimenoEigenfunctions}, \cite[\S 3 and
5]{ShimenoPlancherel}, \cite{HoOl14}.
Indeed, $\D_\ell(m)=\delta_\ell(\D_\ell(G/K))$, where $\delta_\ell$ denotes the $\tau_\ell$-radial component in the sense of Harish-Chandra \cite{HC60}; see \cite[Section 3.6]{OS17}. More specifically, 
\begin{equation}
\label{eq:Lell-deltaell}
D_{\ell, p_L} =\delta_\ell(\Omega)-\tau_{-\ell}(\Omega_\fm)\,,
\end{equation}
where $\Omega$ and $-\Omega_\fm$ are the Casimir operator on $\mathfrak{g}_\C$ and 
$\mathfrak{m}_\C$, respectively.
See \cite[Lemma 2.4]{ShimenoPlancherel} and \cite[Section 3.5]{OS17}. 
Hence
\begin{equation}
\label{phillambda-as-F}
    \varphi_{\ell,\lambda}|_\mathfrak{a}=F_{\ell,\l}(m)=u^{-\ell}F_\lambda(m(\ell))\,.
\end{equation}
Since $m_\rml=1$, the classical symmetry $\varphi_{\ell,\l}=\varphi_{-\ell,\l}$ is a special case of the symmetry $F_{\ell,\l}=F_{-\ell+m_\rml-1,\l}$ from \eqref{eq:Fellm-even-ell}.

In the usual parametrization, the trivial representation of $K$ corresponds to $\ell=0$. In this case, $C^\infty(G//K;\tau_0)$ is the space of the smooth $K$-invariant functions on $G/K$, $\D_0(G/K)$ coincides with the algebra $\D(G/K)$ of $G$-invariant differential operators on $G/K$, and the $\tau_{0}$-spherical functions $\varphi_{0,\lambda}$ are precisely Harish-Chandra's spherical functions $\varphi_{\lambda}$ on $G/K$.

The following proposition summarizes the basic properties of the
$\tau_{\ell}$-spherical functions. As in the $K$-biinvariant case, they can be explicitly given by an integral formula, which extends to arbitrary real $\ell$'s the classical integral formula by Harish-Chandra's for the spherical function $\varphi_\l=\varphi_{0,\lambda}$ on $G/K$.

\begin{Prop}
For $\ell \in \mathbb R$ and $\lambda \in \frakacs$ set
\begin{equation}
\label{eq:integral-formula}
 \varphi_{\ell,\lambda}(g) = \int_{K/Z(G)} e^{(\lambda -
\rho)(H(gk))} \tau_{\ell}(k\kappa(gk)^{-1})~dk, \qquad g \in G.
\end{equation}
Then the set of functions $\varphi_{\ell, \lambda}$, $\lambda \in
\frakacs$ exhausts the class of (elementary) $\tau_{\ell}$-spherical functions on $G.$ Two such functions $\varphi_{\ell,\lambda}$
and $\varphi_{\ell,\mu}$ are equal if and only if $\mu =
w\lambda$ for some $w \in W.$ Moreover, for a fixed $g \in G,$
$\varphi_{\ell, \lambda}(g)$ is holomorphic in $(\lambda, \ell)
\in \frakacs \times \mathbb C$. Furthermore, $\varphi_{\ell, \lambda}=\varphi_{-\ell, \lambda}$.
\end{Prop}
\begin{proof}
See \cite[Proposition 6.1]{ShimenoEigenfunctions}.
\end{proof}

The integral formula together with the properties of Harish-Chandra's spherical functions $\varphi_\l$ ($\ell = 0$ case) automatically
implies several of the properties of the $\tau_{\ell}$-spherical functions we are studying in this paper. 
Nevertheless, others, such as their positivity or the full characterization of the bounded spherical functions, do not seem to follow from \eqref{eq:integral-formula}. We collect all properties in the following corollary. 
Recall that $\ell_{\max}=\frac{m_\rms}{2}+1$ because $m_\rml=1$.


\begin{Cor}\label{Cor:estimates-spherical}
Let $\ell\in \R,$ then we have:
\begin{enumerate}
\item $\varphi_{\ell, \l}$ is real valued for $\l\in\fa^*$.
\item $\varphi_{\ell, \l}|_{\mathfrak a}$ is positive for $\l \in \fa^*$ and $|\ell|\leq \ell_{\max}.$
\item $|\varphi_{\ell,\l}|\leq \varphi_{\Re\l}$ for $\l\in\fa_\C^*$.
\item $|\varphi_{\ell,\l}(m)|\leq 1$ for $\l \in
C(\rho(m)) + i \fa^\ast$, where $C(\rho(m))$ is the convex hull of
the set $\{w\rho(m): w \in W\}.$
\item Let
$|\ell|<\ell_{\max}$. The $\tau_{\ell}$-spherical function $\varphi_{\ell, \l}(m)$ on $G/K$ is bounded if and only if $\lambda \in C(\rho(m)) + i \fa^\ast,$ where $C(\rho(m))$ is the convex hull of
the set $\{w\rho(m):~w \in W\}.$ Moreover, $|\varphi_{\ell, \l}(m;x)| \leq 1$ for all $\l \in
C(\rho(m)) + i \fa^\ast$ and $x \in \fa$.
\end{enumerate}
\end{Cor}
\begin{proof}
The first property follows from the equality $\varphi_{\ell,\l}=\varphi_{-\ell,\l}$ because $\tau_\ell+\tau_{-\ell}$ is real valued. 
For part (2), by the symmetry in $\ell$, 
we can suppose that $\ell\geq 0$. So $m(\ell)\in \mathcal{M}_1$. The positivity 
of $\varphi_{\ell,\l}|_{\mathfrak{a}}$ follows then from Propositions \ref{prop:positivity},(a) and \ref{prop:positivity-estimates-M3}. Part (3) is a consequence of the integral formula, and the fact that $\tau_\ell$ is a unitary character. Part (4) follows from (3) and the theorem by Helgason and Johnson characterizing the parameters $\lambda$ for which the Harish-Chandra's spherical functions $\varphi_\lambda$ are bounded; see \cite{HeJo}. 
Finally, for (5), we can suppose that $\ell\geq 0$ so that $m(\ell)\in \mathcal{M}_1$. Then (5) is an immediate consequence of (3) together with the second part of Theorem \ref{thm:bdd}. 
\end{proof}


\subsection{$\tau$-spherical functions for other small $K$-types $\tau$}
\label{subsection:smallKtype}

We start by briefly recalling the main result from \cite{OS17}, which expresses spherical functions associated to small $K$-types as hypergeometric functions multiplied with an explicit function involving hyperbolic sines and cosines. We refer to \cite{OS17} for the proofs of the properties mentioned here and for further information.

Let $G$ be a non-compact connected real semisimple Lie group with finite center. Let $G = KAN$ be an Iwasawa decomposition and let $M$ be the centralizer of $A$ in $K.$
Recall that a $K$-type $(\tau, V_\tau)$ (i.e. an irreducible representation of $K$) is called \textit{small} if it is irreducible as an $M$-module. We denote by $\phi_\lambda^\tau$
the $\tau$-spherical function on $G$ with spectral parameter $\lambda\in \mathfrak{a}^*$. Moreover, with a slight abuse of notation, we denote by $\phi_\lambda^\tau|_\mathfrak{a}$ the $W$-invariant scalar function such that
the restriction of $\phi_\lambda^\tau$ to $A\equiv\mathfrak{a}$ 
is $\phi_\lambda^\tau|_\mathfrak{a} \id$, where 
$\id$ denotes the identity operator on $V_\tau$.

\begin{Thm} \textup{(}See \cite[Theorem 1.6]{OS17}\textup{)}
\label{thm:smalltypes}
Suppose $(\tau, V)$ is a small $K$-type of a non-compact real simple Lie group $G$ with finite center. If $G$ is a simply-connected split Lie group $\widetilde{G}_2$ of type $G_2$, we further suppose that $\tau$ is not the small type $\pi_2$ specified in Theorem 2.2 of \cite{OS17}. Then there exists a root system $\Sigma(\tau)$ in $\mathfrak{a}^\ast$ and a multiplicity function $m(\tau)$ on $\Sigma(\tau)$ such that 
$$
\phi_\lambda^\tau|_{\mathfrak{a}}= \widetilde{\delta}^{-\frac{1}{2}}_{G/K}~\widetilde{\delta}_{\Sigma(\tau)}(m(\tau))^{\frac{1}{2}}~F_\lambda(\Sigma(\tau), m(\tau))
$$ 
for all $\lambda \in \mathfrak{a}_{\mathbb C}^\ast.$
\end{Thm}

In the above, $F_\lambda(\Sigma(\tau),m(\tau))$ stands for the Heckman-Opdam hypergeometric function associated with the triple $(\mathfrak{a},\Sigma(\tau), m(\tau))$. Here, we recall that we use the symmetric space notation and hence our $\Sigma(\tau)$ and $m(\tau)$ are related to
 $\Sigma^\tau$ and ${\bf{k}}^\tau$ of \cite{OS17} by $\Sigma^\tau = 2 \Sigma(\tau)$ and ${\bf{k}}^\tau_{2\alpha} = \frac{m_\alpha(\tau)}{2}.$ For any root system $(\Sigma, m)$, the function $\widetilde{\delta}_\Sigma(m)$  is defined as:
\begin{equation}
\label{eqn:delta-function}
\widetilde{\delta}_\Sigma(m) = \prod_{\alpha \in \Sigma^+} \left | \frac{\sinh \alpha}{|\alpha|} \right |^{m_\alpha}
\end{equation} 
and $\widetilde{\delta}_{G/K}$ corresponds to the root system for $G/K.$

The classification of the small $K$-types $\tau$ and the concrete choices (which are one or two) of pairs $(\Sigma(\tau),m(\tau))$ occurring in Theorem \ref{thm:smalltypes} can be found in \cite[Section2]{OS17}. 
Clearly, the trivial $K$-type is always small. If $G/K$ is Hermitian, a $K$-type $\tau$ is small if and only if it is a one-dimensional unitary character $\tau_\ell$, as in subsection \ref{subsection:tauell-spherical}. 
All other small $K$-types have dimension bigger than one. 

As in the scalar and the Hermitian cases, the $\tau$-spherical functions corresponding to a small $K$-type can be represented by the integral formula
\begin{equation}
\label{integral-formula-tau}
\varphi_\lambda^\tau(g)=\int_K e^{(\lambda-\rho_{G/K})(H(gk))} \tau(k\kappa(gk))^{-1} \, dk\,,
\end{equation}
see \cite[(3.7)]{OS17}, where $\rho_{G/K}$ is computed from the root system for $G/K$. Formula \eqref{integral-formula-tau} is a special instance of \cite[9.1.5, p. 300]{Warner2} and \cite[(3.9)]{Ca97}.  

Based on the case-by-case analysis of \cite[Section 2]{OS17}, we now show that every $\tau$-spherical function for a symmetric space $G/K$ with restricted root system of type $BC$ and for which $\dim\tau\geq 2$ is a $(\ell,\well)$-hypergeometric function for suitable choices of a multiplicity function $m=(m_\rms,m_\rmm,m_\rml)\in \mathcal{M}_+$, defined on a root subsystem of $\Sigma(\tau)$, and of values $\ell$ and $\well$. It follows, in particular, that $m$ is in general different from both the original root multiplicity function of $G/K$ and $m(\tau)$. 

This identification, together with \eqref{integral-formula-tau}, allows us to 
prove an analogue of 
Corollary \ref{Cor:estimates-spherical}, in which one has to replace $\varphi_{\ell,\lambda}$ with 
$\varphi_\lambda^\tau|_{\mathfrak{a}}$ and the condition $|\ell|\leq \ell_{\max}$ by the conditions $\ell_{\min}(m)-1\leq \ell_{\max}(m)$ and $\well\geq -\frac{m_\rmm}{2}$. Notice that the condition on $\ell$ is equivalent to the more symmetric condition 
$\big|\ell-(m_\rml-1)/2\big|\leq (m_\rms+m_\rml+1)/2$. The precise statement depends on the specific values of 
$\ell_{\min}(m)=-m_\rms/2$ and $\ell_{\max}(m)=
m_\rms/2+m_\rml$ determined below. We leave to the reader the task of writing down the precise statement in each case. 
In fact, the symmetry \eqref{eq:Fellm-even-ell} shows the multiplicity functions $m(\ell,\well)$ and  $m(-\ell+m_\rml-1,\well)$ are different if 
$\ell\neq \frac{m_\rml-1}{2}$ but yield
to the same $\tau$-spherical function.

 As remarked in \cite[Section 2]{OS17}, the classification of $K$-types can be given at the level of Lie algebras of non-compact type since a small $K$-type of $G$ always lifts to that of a finite cover of $G$. We keep the 
convention that a root multiplicity $0$ is equivalent to the fact that the corresponding root does not belong to the fixed root system. 


\subsubsection{The case of $\mathfrak{g} = \mathfrak{sp}(p, 1)$, $p\geq 2$.}
Let $G=\Sp(p,1)$ (which is simply connected) and $K=\Sp(p)\times \Sp(1)$. The small $K$-types are precisely the representations of the form $\tau_n=1\otimes \tau'_n$, the tensor product of the trivial representation of $\Sp(p)$ and the $n$-dimensional irreducible representation of $\Sp(1)=\SU(2)$ with $n\in \N$. The original root system is given by 
$\Sigma = \{\pm \alpha, \pm 2\alpha \}$ (of type $BC_1$),
with multiplicities $m_{\pm \alpha} = 4(p-1)$ and $m_{\pm 2\alpha} = 3.$ 
For the small type $\tau_n$, we have 
$\Sigma(\tau_n) = \Sigma$ and $m(\tau_n)_\pm= (4p-2\pm 2n, 1\mp 2n )$; see \cite[Section 2.3]{OS17}. The associated spherical function is given by 
\begin{equation}
\label{eq:sp-case}
\phi_\lambda^{\tau_n}|_\mathfrak{a} = 
(\cosh \alpha)^{-1\mp n}~F_\lambda(\Sigma, m(\tau_n)_\pm)\,.
\end{equation}
Notice that $m(\tau_n)_+=m(\ell_n)$ for $m=(4(p-1),3)\in \mathcal{M}_+$ and $\ell_n=n+1$. Moreover,
$\ell_{\min} = -2(p-1)$ and
$\ell_{\max} = 2p+1$. So $m(\tau_n)_+ \in \mathcal{M}_+ \cup \mathcal{M}_3$ provided $n \leq  2p$ and $m(\tau_n)_+ \in \mathcal{M}_1$ if $n<2p$. Notice that $\rho(m)=\rho_{G/K}=(2p+1)\alpha$.

Since $m(\tau_n)_-=m(-\ell+m_\rml-1)$, where $-\ell+m_\rml-1=-n+1$, the fact that the two multiplicity functions lead to the same $\tau_n$-spherical function
is a special instance of the symmetry \eqref{eq:Fellm-even-ell}.


\subsubsection{The case of $\mathfrak{g}=\mathfrak{so}(2r,1)$, $r\geq 2$}
\label{SO2r1}
Let $G=\Spin(2r,1)$ be the double cover of $\SO(2r,1)$ and $K=\Spin(2r)$. Then $G$ is simply connected. The non-trivial small $K$-types are the irreducible representations $\tau_s^\pm$ with highest weight $(s/2,\dots, s/2,\pm s/2)$ in standard notation, where $s\in \N$. The case $s=1$ corresponds to the 
positive and negative spin representations and the corresponding $\tau$-spherical analysis was studied in \cite{CaPe}. 

The root system of $G/K$, say $\{\pm \alpha\}$, is of type $A_1$ and $m_\alpha=2r-1$. According to \cite[Section  2.4]{OS17}, $\Sigma(\tau_s^\pm)=\{\pm \alpha/2,\pm \alpha\}$, $m(\tau_s^\pm)=(-2s,2r+2s-1)$ and
$$
\varphi_\lambda^{\tau_s^\pm}|_{\mathfrak{a}}=\cosh^s\big(\frac{\alpha}{2}\big) \; F_\lambda(\Sigma(\tau_s^\pm), m(\tau_s^\pm))\,.
$$
Then $m(\tau_s^\pm)=m(\ell_s)$ for $m=(0,2r-1)\in \mathcal{M}_+$ on $\Sigma(\tau_s^\pm)$ and $\ell_s=-s$. Notice that $\ell_{s,{\min}}=0$ and $\ell_{s,{\max}}=2r-1$. The symmetry \eqref{eq:Fellm-even-ell} gives (with $\well=0$) 
$F_{-s,\lambda}(m)=F_{s+2r-2,\lambda}(m)$. For $s=1$, the symmetric of $\ell_s=-1$ is $2r-1$ and $m(2r-1)\in \mathcal {M}_3$ (but $\notin \mathcal {M}_1$). Observe also that
$\rho(m)=\rho_{G/K}=(r-1/2)\alpha$.


\subsubsection{The case of $\mathfrak{g}=\mathfrak{so}(p,q)$, $p > q\geq 3$}
\label{SOpq}
Let $G$ be the double cover of $\Spin(p,q)$ (which is simply connected) and $K=\Spin(p)\times \Spin(q)$. The root systems of $G/K$ is $\{\pm \beta_j; 1 \leq j \leq q\} \cup \{\pm \beta_j\pm \beta_i; 1 \leq i< j \leq q\}$, with root multiplicities 
$m_{\beta_j}=p-q$ and $m_{\beta_j\pm\beta_i}=1$.
According to \cite[Section 2.5]{OS17}, the small $K$-types are of two forms:
\begin{itemize}
    \item[(i)] $\tau=1\otimes \sigma$, where $\sigma$ is the spin representation of $\Spin(q)$ if $q$ is odd, and either of the two spin representations of $\Spin(q)$ if $q$ is even.
    \item[(ii)] $\tau=\sigma\otimes 1$, where $\sigma$ is either of the two spin representations of $\Spin(p)$ if $p$ is even and $q$ is odd.
\end{itemize}
In case (i), $\Sigma(\tau)=
\{\pm\beta_j; 1\leq j\leq q\}\cup \{\frac{\pm \beta_j\pm \beta i}{2}; 1\leq i < j \leq q\}$ with 
$m(\tau)=(0,1,p-q)$. In case (ii), $\Sigma(\tau)=
\{\pm\frac{\beta_j}{2}; 1\leq j\leq q\}\cup \{\frac{\pm \beta_j\pm \beta i}{2}; 1\leq i < j \leq q\} \cup \{\pm\beta_j; 1\leq j\leq q\}$ with 
$m(\tau)=(2(p-q), 1,-(p-q))$.
Furthermore, $\phi_\lambda^\tau|_{\mathfrak{a}}$ equals
$$
\prod_{1\leq i\leq j\leq q} \Big( \cosh\big(\frac{\beta_j+ \beta_i}{2}\big)
    \cosh\big(\frac{\beta_j- \beta_i}{2}\big)
    \Big)^{-1/2} F_\lambda(\Sigma(\tau),  m(\tau)), \qquad \text{in case (i)},
$$
and
\begin{multline*}
\prod_{j=1}^q \Big( \cosh\big(\frac{\beta_j}{2}\big)\Big)^{-(p-q)}
    \prod_{1\leq i\leq j\leq q} \Big( \cosh\big(\frac{\beta_j+ \beta_i}{2}\big)
    \cosh\big(\frac{\beta_j- \beta_i}{2}\big)
    \Big)^{-1/2} F_\lambda(\Sigma(\tau),  m(\tau)),\\
\text{in case (ii)}\,.
\end{multline*}
The two cases come from the multiplicity on $\Sigma(\tau)$ equal to
$m=(0,0,p-q)\in \mathcal{M}_+$ and 
\begin{itemize}
\item[]
$(\ell,\well)=(0,1/2)$ in case (i), so that $m(\ell,\well)=(0,1,p-q)$\,; 
\item[]
$(\ell,\well)=(p-q,1/2)$ in case (ii), so that $m(\ell,\well)=(2(p-q),1,-(p-q))$\,.
\end{itemize}
Notice that $\ell_{\min}(m)=0$ and $\ell_{\max}(m)=m_\rml=p-q \geq 1$.
Hence, in both cases, $m(\ell,\well)\in \mathcal{M}_+\cup \mathcal{M}_3$ and $m(\ell,\well)\notin \mathcal{M}_1$. The symmetry \eqref{eq:Fellm-even-ell} with $\ell+m_\rml-1=-\ell+p-q-1$ yields the equalities
\begin{eqnarray*}
&&F_{0,1/2,\lambda}(m)=
F_{p-q-1,1/2,\lambda}(m)\,, \qquad\text{in case (i)}\,,\\
&&F_{p-q,1/2,\lambda}(m)=
F_{-1,1/2,\lambda}(m)\,, \qquad\text{in case (ii)}\,.
\end{eqnarray*}
Since $m_\rml=p-q \geq 1$, we see for instance that Theorem \ref{thm:bdd} applies to case (i). Notice also that $\rho_{G/K}=\rho(m(2\well))=\sum_{j=1}^q \big(\frac{p-q}{2}+(j-1)\big) e_j$.

\appendix

\section{The Heckman-Opdam hypergeometric functions as functions of their parameters}
\label{appendix:A}

Let $\fa$ be an $r$-dimensional Euclidian real vector space, with an inner product
$\inner{\cdot}{\cdot}$, and let $\Sigma$ be a root system $\Sigma$ in $\fa^*$ of Weyl group $W$.
Let $\mathcal M_\C$ denote the set of complex-valued multiplicity functions $m=\{m_\alpha\}$ on $\Sigma$.
(Hence $\mathcal M_\C\equiv \C^d$ where $d$ is the number of Weyl group orbits in $\Sigma$.)

In this appendix we collect the regularity properties of the (symmetric and nonsymmetric) hypergeometric
functions $F_\l(m;x)$ and $G_\l(m;x)$ as functions of $(m,x,\l)\in \mathcal M_\C\times \fa_\C \times \fa_\C^*$.
Most of the results are known, but scattered in the literature.

Recall the Gindikin-Karpelevich formula for Harish-Chandra's $c$-function:
\begin{equation}\label{eq:c}
c(m;\l)=\frac{\wt{c}(m;\l)}{\wt{c}(m;\rho(m))}  \qquad (\l\in\fa_\C^*)\,,
\end{equation}
where $\wt{c}(m;\l)$ is given in terms of the positive indivisible roots $\a \in \Sigma_{\rm i}^+$ by
\begin{equation} \label{eq:wtc}
\wt{c}(m;\l)=\prod_{\a \in \Sigma_{\rm i}^+} \frac{2^{-\la} \; \Gamma(\la)}
{\Gamma\Big(\frac{\la}{2}+\frac{m_\a}{4}+\frac{1}{2}\Big)
\Gamma\Big(\frac{\la}{2}+\frac{m_\a}{4}+\frac{m_{2\a}}{2}\Big)}
\end{equation}
and $\Gamma$ is the classical gamma function. Set (see \cite[p. 196]{OpdamIV})
$$
\mathcal{M}_{F, {\rm reg}}=\{m \in \mathcal{M}_\C: \text{$\frac{1}{\wt{c}(m;\rho(m))}$ is not singular}\}
$$
For an irreducible representation $\delta\in \widehat{W}$ and $m\in \mathcal{M}_\C$, let
$\varepsilon_\delta(m)=\sum_{\a\in\Sigma^+} m_\a (1-\chi_\delta(r_\alpha)/\chi_\delta(\id))$,
where $r_\alpha$ is the reflection across $\Ker(\alpha)$ and $\chi_\delta$ is the character of $\delta$.
Let $d_\delta$ be the lowest embedding degree of $\delta$ in $\C[\fa_\C]$, and
set (see \cite[Definition 3.13]{OpdamActa})
$$
\mathcal{M}_{G, {\rm reg}}=\{m \in \mathcal{M}: \text{$\Re(\varepsilon_\delta(m))+d_\delta>0$ for all $\delta\in
\widehat{W}\setminus \{{\rm triv}\}$}\}\,.
$$
Finally, let
$$
\Omega=\{x \in \fa: \text{$|\alpha(x)|<\pi$ for all $\alpha\in \Sigma^+$}\}\,.
$$
The regularity properties of the (symmetric and nonsymmetric) functions are summarized in the following theorem.

\begin{Thm}
\label{thm:reg-prop}
The hypergeometric function $F_\lambda(m;x)$ is  holomorphic in $(m,x,\l)\in
\mathcal{M}_{F, {\rm reg}} \times (\fa+i\Omega) \times \fa_\C^*$ and satisfies
$$F_{\l}(m;wx) = F_{\l}(m;x)=F_{w\l}(m;x) \qquad (m \in \mathcal{M}_{F, {\rm reg}}, \, w\in W, \, x \in \fa+i\Omega, \l\in \fa_\C^* )\,.$$
The (non-symmetric) hypergeometric function $G_\lambda(m;x)$ is a holomorphic function of $(m,x,\l)\in
\mathcal{M}_{G, {\rm reg}} \times (\fa+i\Omega) \times \fa_\C^*$.
\end{Thm}
\begin{proof}
For $\fa+i\Omega$ replaced by a $W$-invariant tubular domain $V$ of $\fa$ in $\fa_\C$, these results are due to Opdam, see \cite[Theorem 2.8]{OpdamIV} and \cite[Theorem 3.15]{OpdamActa}. See also \cite[Theorem 4.4.2]{HS94}. The remark that the maximal tubular domain $V$ is $\fa+i\Omega$ was made by Jacques Faraut, see \cite[Remark 3.17]{BOP}.
\end{proof}

\begin{Prop}
Set
\begin{eqnarray}
\label{eq:MC+}
\mathcal{M}_{\C,+}&=&\{m\in \mathcal{M}_\C: \text{$\Re(m_\a)\geq 0$ for every $\alpha\in \Sigma^+$}\}\\
\label{eq:MC0}
\mathcal{M}_{\C,0}&=&\{m\in \mathcal{M}_\C: \text{$\Re(m_\a+m_{2\a})\geq 0$ for every $\alpha\in \Sigma_{\rm i}^+$}\}\,.
\end{eqnarray}
Then
$$\mathcal{M}_{\C,+}\subset \mathcal{M}_{\C,0} \subset
\mathcal{M}_{F, {\rm reg}} \cap \mathcal{M}_{G, {\rm reg}}\,.$$
Moreover, $\mathcal{M}_{\C,0}$ is stable under deformations by $\ell\in \C$ as in \eqref{eq:mult-l-tl-gen}:
$m\in \mathcal{M}_{\C,0}$ if and only if $m(\ell)\in \mathcal{M}_{\C,0}$.
\end{Prop}
\begin{proof}
The inclusion $\mathcal{M}_{\C,0} \subset \mathcal{M}_{F, {\rm reg}} $ was observed in \cite[Remark 4.4.3]{HS94} and
follows from the properties of the Gamma function. Notice that $\Re(\rho(m)_\alpha)\geq \Re(\frac{m_\a}{2}+m_{2\a})$ for $\a\in \Sigma_{\rm i}^+$, with simplifications in the factor of $c(m;\rho(m))$ corresponding to $\a$ in case of
equality with $\frac{m_\a}{2}+m_{2\a}$ real.
For the inclusion  $\mathcal{M}_{\C,0} \subset \mathcal{M}_{G, {\rm reg}}$, observe that
$\varepsilon_\delta(m)=\sum_{\a\in\Sigma_{\rm i}^+} (m_\a+m_{2\a}) (1-\chi_\delta(r_\alpha)/\chi_\delta(\id))$,
with $\chi_\delta(r_\alpha)/\chi_\delta(\id))\in [-1,1]$.
\end{proof}

\section{Some computations}
\label{appendix:B}

In this appendix, we prove two inequalities stated in subsection \ref{subsection:sharp}. 
We begin showing that for every
$\lambda \in\mathfrak{a}^*$ and $\xi \in \mathfrak{a}$
\begin{equation}
\label{eq:inequality}
\partial_\xi \Big( e^{K_\xi \frac{\inner{\xi}{\cdot}}{|\xi|^2}} F_{\lambda}(m;\cdot) \Big) \geq 0\,.
\end{equation}
Firstly, as in the proof of Lemma 3.4 in \cite{SchThese}, for every $x\in \overline{\mathfrak{a}^+}$, $\xi\in \mathfrak{a}$ and $\lambda\in \mathfrak{a}^*$ we have
$$
\partial_\xi F_\lambda(x)=\frac{1}{|W|}\sum_{w\in W} (\lambda-\rho(m))(w\xi)G_\lambda(wx)
$$
(where $\xi$ is fixed and acts on the $x$-variable).
Here $G_\lambda(wx)\geq 0$ for $m\in \mathcal{M}_3$ (as in \cite{SchThese} for $m\in \mathcal{M}_+$).
Set $K_\xi=\max_{w\in W} (\rho(m)-\lambda)(w\xi)$. So
$K_\xi=-\min_{w\in W} (\lambda-\rho(m))(w\xi)$ and
$$
\partial_\xi F_\lambda(x)\geq -K_\xi \frac{1}{|W|}\sum_{w\in W} G_\lambda(wx)=-K_\xi
F_\lambda(x)\,.$$
Since $x\mapsto F_\lambda(x)$ and $\xi\mapsto K_\xi$ are $W$-invariant, the last inequality extends to all $x\in \mathfrak{a}$.
Hence, for every $x,\xi\in \mathfrak{a}$ and every $\lambda\in \mathfrak{a}^*$ we have
$$
\partial_\xi \big( e^{K_\xi \frac{\inner{\xi}{\cdot}}{|\xi|^2}} F_{\lambda}(m;\cdot) \big)
=e^{K_\xi \frac{\inner{\xi}{\cdot}}{|\xi|^2}} \Big( \partial_\xi F_\lambda+K_\xi F_\lambda\Big) \geq 0\,.
$$
The previous argument is essentially the one leading to (8) in \cite[Lemma 3.3]{SchThese}. 
\medskip

We pass from \eqref{eq:inequality} to the inequality \eqref{eq:subadd-Schapira} in Lemma \ref{lemma:subadd-Schapira} by using the following arguments, which are repeatedly used in \cite{SchThese}. This is why we say that Lemma \ref{lemma:subadd-Schapira} is implicit in \cite{SchThese}.

Let $\lambda \in \mathfrak{a}^*$ and $x\in \mathfrak{a}$ be fixed.  Set
$$f(t)=e^{K_\xi \frac{\inner{\xi}{x+t\xi}}{|\xi|^2}} F_{\lambda}(m;x+t\xi)\,.$$
By Lemma 3.2 in \cite{SchThese} and since
$$
f'_d(t)=\left.\frac{d}{dt}\right|_{t\searrow 0} f(t)=
\partial_\xi \Big( e^{K_\xi \frac{\inner{\xi}{\cdot}}{|\xi|^2}} F_{\lambda}(m;\cdot) \Big) (x)\geq 0\,,
$$
we conclude that $f(t)$ is increasing on $[0,+\infty[$ and hence
$$
e^{K_\xi \frac{\inner{\xi}{x}}{|\xi|^2}} e^{K_\xi t} F_{\lambda}(m;x+t\xi)=f(t)\geq f(0)=
e^{K_\xi \frac{\inner{\xi}{x}}{|\xi|^2}} F_{\lambda}(m;x)
$$
i.e.
\begin{equation}
\label{eq:inequality2}
e^{K_\xi t} F_{\lambda}(m;x+t\xi)\geq F_{\lambda}(m;x)
\end{equation}
for all $t\geq 0$.
Choose first $\xi=x_1\in \mathfrak{a}$ and $t=1$. Then
\eqref{eq:inequality2} gives
\begin{equation}
\label{half-inequalityLemma2.9-1}
e^{\max\limits_{w\in W} (\rho(m)-\lambda)(wx_1)} F_{\lambda}(m; x+x_1)
\leq F_\lambda(m; x)\,.
\end{equation}
Replace $x$ by $x+x_1$ and choose $\xi=-x_1$ and $t=1$. Then
\eqref{eq:inequality2} gives
\begin{equation}
\label{half-inequalityLemma2.9-2}
e^{-\min\limits_{w\in W} (\rho(m)-\lambda)(wx_1)} F_{\lambda}(m;x)
\geq F_\lambda(m; x+x_1)\,,
\end{equation}
as required.


\end{document}